\documentclass{amsart}
\usepackage{amsfonts, amssymb, wasysym}
\usepackage{multicol}
\usepackage{amsmath}
\usepackage{latexsym}
\usepackage{mathrsfs}
\usepackage{epsfig}
\usepackage{graphicx}
\usepackage{tikz}
\usetikzlibrary{positioning}
\usepackage{hyperref}
\hypersetup{colorlinks,citecolor=blue,linkcolor=blue}
\usepackage{slashed}
\usepackage[textsize=scriptsize]{todonotes}
\usepackage{hyperref}

\def\Z{\mathbb{Z}}

\def\CC {\mathcal{C}}

\def\HH {\mathcal{H}}

\def\SS{\mathcal{S}}
\def\Sp{{\SS}^{+}}
\def\Sm{{\SS}^{-}}
\def\Spm{{\SS}^{\pm}}

\newcommand{\dirac}{\slashed{D}}
\def\dplus{{\dirac}^+}
\def\dminus{{\dirac}^{-}}

\def\Ahat {\widehat{\mathcal{A}}}

\def\delbar{\bar\partial}

\newcommand{\tr}{\operatorname{tr}}

\newtheorem{theorem}{Theorem}[section]
\newtheorem{proposition}[theorem]{Proposition}
\newtheorem{lemma}[theorem]{Lemma}
\newtheorem{corollary}[theorem]{Corollary}

\usepackage{cyr}

\DeclareSymbolFont{mcyr}{OT1}{wncyr}{m}{n}
\DeclareMathSymbol\Lob{\mathop}{mcyr}{76}
\newcommand{\integers}{{\mathbb Z}}
\newcommand{\realnos}{{\mathbb R}}

\newcommand{\complexnos}{{\mathbb C}}
\newcommand{\quaternions}{{\mathbb H}}

\def\Kappa{{\rm K}}

\usepackage{graphicx}
\makeatletter
\DeclareRobustCommand*\uell{\mathpalette\@uell\relax}
\newcommand*\@uell[2]{
  \setbox0=\hbox{$#1\ell$}
  \setbox1=\hbox{\rotatebox{10}{$#1\ell$}}
  \dimen0=\wd0 \advance\dimen0 by -\wd1 \divide\dimen0 by 2
  \mathord{\lower 0.1ex \hbox{\kern\dimen0\unhbox1\kern\dimen0}}
}

\begin{document}

\title{Harmonic spinors on the Davis hyperbolic 4-manifold}

\author{John G. Ratcliffe, Daniel Ruberman and Steven T. Tschantz}

\address{Department of Mathematics, Brandeis University, Waltham, MA 02445}
\address{Department of Mathematics, Vanderbilt University, Nashville, TN 37240
\vspace{.1in}}

\email{j.g.ratcliffe@vanderbilt.edu}
\email{ruberman@brandeis.edu}
\email{steven.tschantz@vanderbilt.edu}


\date{}


\begin{abstract}
In this paper we use the $G$-spin theorem to show that the Davis hyperbolic 4-manifold admits harmonic spinors. 
This is the first example of a closed hyperbolic $4$-manifold that admits harmonic spinors.
We also explicitly describe the Spinor bundle of a spin hyperbolic 2- or 4-manifold and 
show how to calculated the subtle sign terms in the $G$-spin theorem for an isometry, with isolated fixed points, of a closed spin hyperbolic 2- or 4-manifold.  
\end{abstract}

\maketitle

\section{Introduction}\label{S:intro} 
The Dirac operator $\dirac$ acting on sections of the spinor bundle $\SS$ 
of a closed spin Riemannian manifold $M$ is one of the fundamental elliptic operators of Riemannian geometry.  
The operator is self-adjoint, and in even dimensions the spinors split as $\SS = \Sp \oplus \Sm$, with 
$\dirac$ interchanging sections of $\Spm$. 
The elements of the kernel $\HH$ of $\dirac$ are call {\it harmonic spinors}, and 
$\HH = \HH^+\oplus\HH^-$ 
where $\HH^\pm$ are, respectively the kernel of the 
{\em chiral} Dirac operator $\dplus : \CC^\infty(\Sp) \to \CC^\infty(\Sm)$ and its adjoint $\dminus$.

The {\it spinor-index} of $M$ is the index of $\dplus$ which is defined to be 
\begin{equation}\label{E:index}  
{\rm Spin}(M) = \dim \HH^+ - \dim\HH^-.
\end{equation}
The spinor-index was determined by Atiyah and Singer~\cite{atiyah-singer:I} to be the $\Ahat$-genus, that is, the integral of a polynomial in the Pontrjagin classes of $M$. Lichnerowicz~\cite{lichnerowicz:spinors} made the seminal observation that for manifolds with positive scalar curvature, both $\HH^+$ and $\HH^-$ are $\{0\}$, and hence the topologically invariant $\Ahat$-genus vanishes.  The method extends to show the vanishing of certain spin bordism invariants ~\cite{hitchin:spinors}. The use of such index-theoretic methods to give topological restrictions on manifolds admitting a metric of positive scalar curvature has been extensive; see the survey~\cite{rosenberg:psc-progress}.

In this paper, we address a converse to the Lichnerowicz result, looking for non-zero harmonic spinors 
on manifolds with negative curvature, in particular for hyperbolic manifolds. 
Note that by a result of Chern~\cite{chern:curvature} the Pontrjagin classes of a hyperbolic manifold vanish, and so the spinor-index is $0$.  Hence the index theorem cannot be used directly to prove the existence of non-zero harmonic spinors. Moreover, for a generic metric on a manifold of dimension at least $3$, the kernel of the Dirac operator is as small as required by the index theorem~\cite{ammann-dahl-humbert:surgery}, so a metric whose associated Dirac operator has non-trivial kernel must be somewhat special.  Hitchin~\cite{hitchin:spinors} used the interpretation of the Dirac operator as a twisted $\delbar$ operator to show that for certain spin structures and hyperbolic metrics on a Riemann surface of genus $g$, the kernel can have dimension as large as 
$
\left\lfloor (g+1)/2\right\rfloor.
$

The main result of this paper is the existence of closed hyperbolic $4$-manifolds for which the kernel of the Dirac operator is non-trivial.  To our knowledge, these are the first known examples of such hyperbolic manifolds in higher dimensions.  
Our primary example is the Davis hyperbolic 4-manifold ~\cite{davis:hyperbolic}, which was shown to have a spin structure in~\cite{ratcliffe-tschantz:davis}. 
\begin{theorem}\label{T:davis}  
Let $M$ be the Davis closed hyperbolic $4$-manifold. 
Then $M$ has a spin structure such that the kernel $\HH$ 
of the Dirac operator $\dirac$ has (complex) dimension at least $20$.
\end{theorem}
By passing to finite covers of $M$, we obtain infinitely many examples.
\begin{corollary}\label{C:covers}  
There are closed spin hyperbolic $4$-manifolds of arbitrarily high volume with non-zero harmonic spinors.
\end{corollary}

Our method is to use the $G$-spin theorem (the $G$-index theorem for the Dirac operator~\cite{atiyah-singer:III}) 
where $G$ is a group of orientation preserving isometries of a Riemannian spin manifold $M$.  Suppose that the action of $G$ lifts to the spinor bundle, and denote by $\hat{g}$ the lift of an element $g\in G$. Then $G$ acts on $\HH^+$ and $\HH^-$, and 
we get two characters of $G$ whose difference in the representation ring $R(G)$ is the spinor-index ${\rm Spin}(G,M)$ of the action. 
The value of ${\rm Spin}(G,M)$ at an element $g$ of $G$ is the $G$-equivariant index
\begin{equation}\label{E:g-spin}  
{\rm Spin}(\hat g,M) := \tr(\hat{g},\HH^+) - \tr(\hat{g},\HH^-).
\end{equation}
\begin{proposition}\label{P:g-spin} 
Let $M$ be a hyperbolic spin manifold of even dimension, and $G$ a spin action as above.  
If for some $g\in G$ we have ${\rm Spin}(\hat g,M) \neq 0$, then $M$ admits non-zero harmonic spinors.
\end{proposition}
\begin{proof}
If $\HH^+ =\{0\}$, then $\HH^- =\{0\}$ as well, by the vanishing of the spinor-index ${\rm Spin}(M)$. 
Hence if ${\rm Spin}(\hat g,M) \neq 0$, then $\HH^+$ must have positive dimension.
\end{proof}
To apply Proposition~\ref{P:g-spin} to the Davis manifold $M$, 
we consider a particular spin structure and cyclic group $G$ generated by an isometry $g$ of order $15$ 
with a lift $\hat{g}$ to the spinor bundle of order 15.  
By understanding the fixed point behavior of $g$ and its powers, we use the $G$-spin theorem 
to determine ${\rm Spin}(G, M)$ in the representation ring $R(G)$. 
This gives the lower bound of $20$ for $\dim \HH$ stated in Theorem~\ref{T:davis}.

The calculation of the local contributions at fixed points to the equivariant index in \eqref{E:g-spin} involves a careful determination of the lift of the $G$-action 
on the manifold to its spinor bundle. 
This is carried out by interpreting spin structures on a hyperbolic manifold in representation-theoretic terms, and is presented in some generality in this paper.
In principle, this same method will work to prove the existence of non-zero harmonic spinors 
on hyperbolic manifolds in other even dimensions. 
We give some examples to show that our method works for hyperbolic 2-manifolds. 

Theorem~\ref{T:davis} also implies that the generic vanishing theorem for $\HH$ in~\cite{ammann-dahl-humbert:surgery} does not have a direct equivariant analogue. 
\begin{corollary}\label{C:equiv}  
There is an open set in the space of $\Z_{15}$-invariant Riemannian metrics on the Davis hyperbolic 4-manifold for which the kernel of the Dirac operator is non-trivial. 
\end{corollary}
A more reasonable equivariant extension of the result of~\cite{ammann-dahl-humbert:surgery} might be that for a spin manifold with a smooth action of a finite group $G$, a generic $G$-invariant metric has kernel and cokernel as small as allowed by the $G$-spin theorem.

\subsection*{Outline}
Our paper is organized as follows:  In \S 2, we give an algebraic characterization of a spin structure on a hyperbolic $n$-manifold. 
In \S 3, we define the spin group ${\rm Spin}^+(n,1)$ in terms of Clifford algebras. 
In \S 4, we show that the complex spin representation $\Delta_{2m}: {\rm Spin}(2m) \to \complexnos(2^{m})$ 
extends to a representation $\Delta_{2m,1}: {\rm Spin}^+(2m,1) \to \complexnos(2^{m})$. 
In \S 5, we describe $\Delta_{2,1}$ in terms of the group ${\rm SU}(1,1;\complexnos)$. 
In \S 6, we describe $\Delta_{4,1}$ in terms of the group ${\rm SU}(1,1;\mathbb{H})$. 
In \S 7, we give a new formulation of the complex spinor bundle of a hyperbolic spin $2m$-manifold.  
In \S 8, we give a refined formulation of the $G$-spin theorem for an isometry of a hyperbolic spin manifold 
with only isolated fixed points. 
In \S 9, we use the $G$-spin theorem to prove the existence of non-zero harmonic spinors 
on two hyperbolic surfaces. 
In \S 10, we use the $G$-spin theorem to prove the existence of non-zero harmonic spinors 
on the Davis hyperbolic 4-manifold.

\subsection*{Acknowledgments}
The question of finding a closed spin hyperbolic manifold of dimension greater than $2$ whose Dirac operator is non-invertible was posed by the second author on Mathoverflow~\cite{MO-spinors}. Thanks to those who submitted comments and suggestions, and to Francesco Lin for some early correspondence on the subject.

The authors thank AIM for providing a congenial environment where our two SQuaREs were able to interact, leading to the current work.  
DR was partially supported by NSF grant DMS-1506328.

\section{Preliminaries}\label{S:prelim} 

\subsection{Hyperbolic $n$-space}
Let $f_n$ be the {\it Lorentzian quadratic form} in $n+1$ variables $x_1, \ldots, x_{n+1}$ given by
$$f_n(x) = x_1^2+ \cdots + x_n^2 - x_{n+1}^2.$$
The hyperboloid model of {\it hyperbolic $n$-space} is 
$$H^n = \{x \in \realnos^{n+1} : f_n(x) = -1 \ \ \hbox{and} \ \ x_{n+1} > 0\}.$$
The {\it orthogonal group} of the quadratic form $f_n$ is defined to be 
$$\mathrm{O}(n,1) = \{A \in \mathrm{GL}(n+1,\realnos) : f_n(Ax) = f_n(x) \ \hbox{for all} \ x \in \realnos^{n+1}\}.$$
Let $\mathrm{O}^+(n,1)$ be the subgroup of $\mathrm{O}(n,1)$ consisting of all $A\in \mathrm{O}(n,1)$ that leave $H^n$ invariant. 
Then $\mathrm{O}^+(n,1)$ has index 2 in $\mathrm{O}(n,1)$. 
Restriction induces an isomorphism from $\mathrm{O}^+(n,1)$ to the group $\mathrm{Isom}(H^n)$ 
of isometries of $H^n$.  
We will identify $\mathrm{O}^+(n,1)$ with $\mathrm{Isom}(H^n)$. 
Let $\mathrm{SO}^+(n,1)$ be the subgroup of $\mathrm{O}^+(n,1)$ of matrices of determinant 1. 
Under the identification of $\mathrm{O}^+(n,1)$ with ${\rm Isom}(H^n)$, the group 
$\mathrm{SO}^+(n,1)$ corresponds to the group of orientation preserving isometries of $H^n$. 

\subsection{Hyperbolic $n$-manifold}
A {\it hyperbolic $n$-manifold} is a complete Riemannian $n$-manifold of constant sectional curvature $-1$. 
As a reference for hyperbolic manifolds, see \cite{R}. 
An $n$-dimensional {\it hyperbolic space-form} is the orbit space $\Gamma\backslash H^n$ of a torsion-free discrete subgroup 
$\Gamma$ of $\mathrm{O}^+(n,1)$.  A hyperbolic space-form  $\Gamma\backslash H^n$ is a hyperbolic $n$-manifold, 
and every hyperbolic $n$-manifold is isometric to a hyperbolic space form  $\Gamma\backslash H^n$. 
The manifold  $\Gamma\backslash H^n$ is orientable if and only if $\Gamma$ is a subgroup of 
$\mathrm{SO}^+(n,1)$. 

Consider the {\it Lorentzian inner product} on $\realnos^{n+1}$ given by 
$$x \circ y = x_1y_1 + \cdots + x_ny_n - x_{n+1}y_{n+1}.$$
We denote $\realnos^{n+1}$ with this inner product by $\realnos^{n,1}$. 
The {\it tangent space} of $H^n$ at a point $x$ of $H^n$ is 
$${\rm T}_x(H^n) = \{y \in \realnos^{n,1}: x \circ y = 0\}.$$
Now ${\rm T}_x(H^n)$ is a $n$-dimensional space-like vector subspace of $\realnos^{n,1}$ for each $x$ in $H^n$, 
and so the Lorentzian inner product on $\realnos^{n,1}$ restricts to a positive definite inner product on ${\rm T}_x(H^n)$.  

The {\it tangent bundle} of $H^n$ is the set 
$${\rm T}(H^n) = \{(x,v) \in H^n \times \realnos^{n,1}: v \in {\rm T}_x(H^n)\}$$
with the subspace topology from $H^n \times \realnos^{n+1}$. 

Let $\Gamma\backslash H^n$ be a hyperbolic space-form. 
Then $\Gamma$ acts diagonally on ${\rm T}(H^n)$; moreover, $\Gamma$ acts freely and discontinuously on ${\rm T}(H^n)$. 
The {\it tangent bundle} ${\rm T}(\Gamma\backslash H^n)$ of $\Gamma\backslash H^n$ is the orbit space 
$\Gamma\backslash {\rm T}(H^n)$. 

\subsection{Orthonormal frame bundle}
The {\it oriented orthonormal frame bundle} of $H^n$ is the set ${\rm F}(H^n)$ of all ordered $(n+1)$-tuples  
$(v_1,\ldots,v_n,x)$  in $(\realnos^{n,1})^{n+1}$, with the subspace topology, such that $x \in H^n$,  
and $\{v_1,\ldots,v_n\}$ is an orthonormal basis for ${\rm T}_x(H^n)$, and ${v_1,\ldots, v_n, x}$ is 
a positively oriented basis of $\realnos^{n+1}$. 
We have the projection map $\pi: {\rm F}(H^n) \to H^n$ defined by $\pi(v_1,\ldots,v_n,x) = x$. 

Define $\xi: {\rm F}(H^n) \to {\rm SO}^+(n,1)$ by $\xi(v_1,\ldots,v_n,x) = A$ where $A$ is the matrix 
whose columns vectors are $v_1, \ldots, v_n, x$. 
Then $\xi$ is a diffeomorphism. 
Let $e_1, \ldots, e_{n+1}$ be the standard basis of $\realnos^{n+1}$. 
Then 
$${\rm F}(H^n) = \{(Ae_1,\ldots, Ae_{n+1}): A \in {\rm SO}^+(n,1)\}.$$
Let $\epsilon: {\rm SO}^+(n,1) \to H^n$ be the evaluation map at $e_{n+1}$. 
Then $\epsilon \xi = \pi$. 

Now ${\rm SO}^+(n,1)$ is a principal ${\rm SO}(n)$-bundle over $H^n$ with projection map $\epsilon$ and $B \in {\rm SO}(n)$ 
acting freely on the right of ${\rm SO}^+(n,1)$ by right multiplication by $\hat B$ where $\hat B \in {\rm SO}^+(n,1)$ 
is the block diagonal matrix with blocks $B$ and $(1)$. 
Moreover ${\rm SO}^+(n,1)$ is a trivial principal ${\rm SO}(n)$-bundle over $H^n$. 
The group ${\rm SO}(n)$ acts freely on the right of $F(H^n)$ by 
$$(Ae_1,\ldots, Ae_{n+1})B = (A\hat Be_1, \ldots, A\hat B e_{n+1})$$
making ${\rm F}(H^n)$ into a principal ${\rm SO}(n)$-bundle over $H^n$ equivalent to  ${\rm SO}^+(n,1)$ via 
the diffeomorphism $\xi$. 

Let $\Gamma\backslash H^n$ be an orientable hyperbolic space-form. 
Then $\Gamma$ acts diagonally on the left of ${\rm F}(H^n)$; 
moreover, $\Gamma$ acts freely and discontinuously on ${\rm F}(H^n)$. 
The {\it orthonormal frame bundle} ${\rm F}(\Gamma\backslash H^n)$ of $\Gamma\backslash H^n$ is the orbit space 
$\Gamma\backslash {\rm F}(H^n)$. 
The left action of $\Gamma$ on ${\rm F}(H^n)$ corresponds to the left action of $\Gamma$ on ${\rm SO}^+(n,1)$ 
by group multiplication. 
We will identify ${\rm F}(H^n)$ with ${\rm SO}^+(n,1)$ 
and ${\rm F}(\Gamma\backslash H^n)$ with $\Gamma\backslash {\rm SO}^+(n,1)$. 
We have that $\Gamma\backslash {\rm SO}^+(n,1)$ is a principal ${\rm SO}(n)$-bundle over $\Gamma\backslash H^n$ 
with right action of ${\rm SO}(n)$ induced by the right action of ${\rm SO}(n)$ on ${\rm SO}^+(n,1)$
and bundle map $\varepsilon: \Gamma\backslash {\rm SO}^+(n,1) \to \Gamma\backslash H^n$ 
defined by $\varepsilon(\Gamma A) = \Gamma \epsilon(A)$. 

It is standard that ${\rm SO}(2)$ is homeomorphic to $S^1$, and so $\pi_1({\rm SO}(2)) \cong \integers$. 
If $n > 2$, then $\pi_1({\rm SO}(n)) \cong \integers/2\integers$, and so $\pi_1({\rm SO}(n))$ has a unique subgroup of index 2 for each $n \geq 2$. 
Hence ${\rm SO}(n)$ has a connected double covering space which is unique up to covering space equivalence. 
We will give a formal definition of ${\rm Spin}(n)$ later in the paper, but for now 
${\rm Spin}(n)$ is a connected double covering space of ${\rm SO}(n)$ for each $n \geq 2$. 
As ${\rm SO}(n)$ is a Lie group, ${\rm Spin}(n)$ is a Lie group
such that the covering projection $\sigma: {\rm Spin}(n) \to {\rm SO}(n)$ is a group homomorphism 
by Theorem 6.11 of \cite{G-H}. 
We denote the nonidentity element of the kernel of $\sigma$ by $-1$. 

\subsection{Spin structure}
Let $\Gamma\backslash H^n$ be an orientable hyperbolic space-form. 
A {\it spin structure} on $\Gamma\backslash H^n$ is 
a double covering projection $\rho: P \to \Gamma\backslash {\rm SO}^+(n,1)$ such that $P$ is a principal ${\rm Spin}(n)$-bundle $P$ over $\Gamma\backslash H^n$,  
with bundle projection $\varepsilon\rho: P \to \Gamma\backslash H^n$,  
and such that if $x$ is in $P$ and $s$ is in ${\rm Spin}(n)$, then $\rho(x s) = \rho(x) \sigma(s)$. 
This last condition implies that $\rho$ double projects each fiber of $\varepsilon\rho: P \to \Gamma\backslash H^n$ 
onto a fiber of $\varepsilon: \Gamma\backslash{\rm SO}^+(n,1) \to \Gamma\backslash H^n$. 

Two spin structures $\rho: P \to \Gamma\backslash {\rm SO}^+(n,1)$ and $\rho': P' \to  \Gamma\backslash {\rm SO}^+(n,1)$ on $\Gamma\backslash H^n$ 
are said to be {\it equivalent} if there is a diffeomorphism $\xi: P \to P'$ such that $\rho' \xi = \rho$ 
and if $x$ is in $P$ and $s$ is in ${\rm Spin}(n)$, then $\xi(x s) = \xi(x) s$. 

The mapping $B \mapsto \hat B$ embeds ${\rm SO}(n)$ isomorphically onto 
a subgroup $\widehat{\rm SO}(n)$ of ${\rm SO}^+(n,1)$ 
which is the fiber of the ${\rm SO}(n)$-bundle projection $\epsilon:  {\rm SO}^+(n,1) \to H^n$ over the point $e_{n+1}$. 
The embedding, $B \mapsto \hat B$, of ${\rm SO}(n)$ into ${\rm SO}^+(n,1)$ is a homotopy equivalence, since $H^n$ is contractible. 
Therefore $\pi_1({\rm SO}^+(2,1)) \cong \integers$ and $\pi_1({\rm SO}^+(n,1)) \cong \integers/2\integers$ for all $n > 2$.  
Hence $\pi_1({\rm SO}^+(n,1))$ has a unique subgroup of index 2 for each $n \geq 2$. 
Therefore ${\rm SO}^+(n,1)$ has a connected double covering space which is unique up to covering space equivalence. 
We will give a formal definition of ${\rm Spin}^+(n,1)$ later in the paper, but for now 
${\rm Spin}^+(n,1)$ is a connected double covering space of ${\rm SO}^+(n,1)$ for each $n \geq 2$. 
As ${\rm SO}^+(n,1)$ is a Lie group, ${\rm Spin}^+(n,1)$ is a Lie group
such that the covering projection $\eta: {\rm Spin}^+(n,1) \to {\rm SO}^+(n,1)$ is a group homomorphism 
by Theorem 6.11 of \cite{G-H}. 
We denote the nonidentity element of the kernel of $\eta$ by $-1$.

The embedding,  $B \mapsto \hat B$, of ${\rm SO}(n)$ into ${\rm SO}^+(n,1)$ lifts to an isomorphic embedding of ${\rm Spin}(n)$ onto the  
subgroup $\widehat{\rm Spin}(n) = \eta^{-1}(\widehat{\rm SO}(n))$ of ${\rm Spin}^+(n,1)$.

The Lie group ${\rm Spin}^+(n,1)$ is a principal ${\rm Spin}(n)$-bundle over $H^n$, 
with bundle projection $\epsilon\eta: {\rm Spin}^+(n,1) \to H^n$ and 
right action of ${\rm Spin}(n)$ on ${\rm Spin}^+(n,1)$ corresponding to right multiplication by $\widehat{\rm Spin}(n)$. 
Moreover if $g$ is in ${\rm Spin}^+(n,1)$ and $s$ is in ${\rm Spin}(n)$, then $\eta(gs) = \eta(g)\sigma(s)$. 
Therefore the double covering projection $\eta: {\rm Spin}^+(n,1) \to {\rm SO}^+(n,1)$ 
is a spin structure on $H^n$. 
Note that ${\rm Spin}^+(n,1)$ is a trivial principal ${\rm Spin}(n)$-bundle over $H^n$, 
since ${\rm SO}^+(n,1)$ is a trivial principal ${\rm SO}(n)$-bundle over $H^n$.

The next theorem is known to experts.  We could not find a proof in the literature, and so we give a proof which relies only on covering space theory. 

\begin{theorem}\label{T:spin-hyperbolic}  
Let $\Gamma$ be a torsion-free discrete subgroup of ${\rm SO}^+(n,1)$, 
and let $\eta: {\rm Spin}^+(n,1) \to {\rm SO}^+(n,1)$ be the double covering epimorphism.  
Then the set of equivalence classes of spin structures on the hyperbolic space-form $\Gamma\backslash H^n$ is  
in one-to-one correspondence with the set of subgroups $\hat \Gamma$ of ${\rm Spin}^+(n,1)$ 
such that $\eta$ maps $\hat \Gamma$ isomorphically onto $\Gamma$. 

If $\hat\Gamma$ is a subgroup of ${\rm Spin}^+(n,1)$ such that $\eta$ maps $\hat \Gamma$ isomorphically onto $\Gamma$, 
then $\hat\Gamma$ corresponds to the equivalence class of the spin structure $\rho: \hat\Gamma\backslash {\rm Spin}^+(n,1) \to \Gamma\backslash {\rm SO}^+(n,1)$ 
induced by $\eta$. 
\end{theorem}

\begin{proof}
Let $\rho: P \to \Gamma\backslash {\rm SO}^+(n,1)$ be a spin structure on $\Gamma\backslash H^n$. 
Then $\rho$ is a double covering projection.  Let $x_0$ be a point of $P$ such that $\rho(x_0) = \Gamma$. 
Then $\rho_\ast\pi_1(P,x_0)$ is a subgroup of $\pi_1(\Gamma\backslash {\rm SO}^+(n,1),\Gamma)$ of index 2. 
Subgroups of index 2 are normal, and so $\rho_\ast\pi_1(P,x_0)$ is a normal subgroup of $\pi_1(\Gamma\backslash {\rm SO}^+(n,1),\Gamma)$. 

Let $\rho': P' \to \Gamma\backslash {\rm SO}^+(n,1)$ be a spin structure on $\Gamma\backslash H^n$ 
such that $\rho$ is equivalent to $\rho'$. 
Then there exists a diffeomorphism $\xi: P \to P'$ such that $\rho'\xi = \rho$. 
Hence $P$ and $P'$ are equivalent covering spaces of $\Gamma\backslash {\rm SO}^+(n,1)$. 
Let $x_0'$ be a point of $P'$ such that $\rho'(x_0') = \Gamma$. 
Then $\rho_\ast'\pi_1(P',x_0') = \rho_\ast\pi_1(P,x_0)$ by Theorem V.6.6 of \cite{M}. 
Hence the equivalence class of the spin structure $\rho$ determines the subgroup $\rho_\ast\pi_1(P,x_0)$ 
of $\pi_1(\Gamma\backslash {\rm SO}^+(n,1),\Gamma)$ of index 2.

Let $G = {\rm Spin}^+(n,1)$. 
If $n = 2$, let $\tilde G$ be a universal covering space of $G$, 
and let $\kappa: \tilde G \to G$ be a covering projection. 
Then $\tilde G$ has a Lie group structure such that $\kappa: \tilde G \to G$ is a group homomorphism 
by Theorem 6.11 of \cite{G-H}.  
If $n > 2$, then $G$ is simply connected, and we let $\tilde G = G$ and $\kappa: \tilde G \to G$ be the identity map. 
Then we have the following commutative diagram whose rows are exact sequences  
$$
\begin{array}{ccccccccc}
1 &\to &\Kappa &\to & \tilde G &{\buildrel \kappa \over \longrightarrow} & G & \to & 1 \\
   &     &\downarrow & & \mid\mid   &   &  \downarrow  \eta& & \\
1 &\to &\Lambda &\to & \tilde G & \longrightarrow & {\rm SO}^+(n,1) & \to & 1. 
\end{array}$$
The group $\Lambda = {\rm Ker}(\eta\kappa)$ acts freely on $\tilde G$ by left multiplication as the group 
of covering transformations of the universal covering projection $\eta\kappa: \tilde G \to {\rm SO}^+(n,1)$. 
Therefore $\Lambda$ is isomorphic to the fundamental group of ${\rm SO}^+(n,1)$. 
Hence $\Lambda$ is infinite cyclic if $n = 2$ and $\Lambda =\{\pm 1\}$ if $n > 2$.  
As $\Lambda =\kappa^{-1}(\{\pm 1\})$, we have that $\Kappa$ 
is a subgroup of $\Lambda$ of index 2. Moreover $\Kappa = \{1\}$ if $n > 2$. 

Let $\tilde\Gamma = \eta^{-1}(\Gamma)$.  
Then $\{\pm 1\}$ is a normal subgroup of $\tilde\Gamma$ and $\{\pm 1\}\backslash \tilde\Gamma \cong \Gamma$. 
Let $\tilde\Gamma' = \kappa^{-1}(\tilde\Gamma)$. 
Let $\pi: {\rm SO}^+(n,1) \to \Gamma\backslash {\rm SO}^+(n,1)$ be the quotient map. 
The group $\tilde\Gamma' $ acts freely on $\tilde G$ by left multiplication. 
The set of orbits $\tilde\Gamma' \backslash\tilde G$ is the set of fibers of 
the universal covering projection $\pi\eta\kappa: \tilde G \to \Gamma\backslash {\rm SO}^+(n,1)$, 
since $\kappa, \eta$, and $\pi$ induce the following bijections
$$\tilde\Gamma' \backslash\tilde G \cong (\Kappa\backslash\tilde\Gamma')\backslash(\Kappa\backslash\tilde G)
\cong \tilde\Gamma\backslash G \cong (\{\pm 1\}\backslash\tilde\Gamma)\backslash(\{\pm 1\}\backslash G)
\cong \Gamma\backslash {\rm SO}^+(n,1). $$
Hence $\tilde\Gamma'$ acts on $\tilde G$ as the group of covering transformations of  $\pi\eta\kappa$. 
Therefore $\tilde\Gamma'\cong\pi_1(\Gamma\backslash {\rm SO}^+(n,1),\Gamma)$. 
Subgroups of index 2 are normal. 
Therefore the equivalence classes of connected double covering spaces of $\Gamma\backslash {\rm SO}^+(n,1)$ correspond 
to the subgroups of $\tilde\Gamma'$ of index 2 by Theorems V.6.6 and V.10.2 of \cite{M}.

The group $\tilde\Gamma$ is the fiber of the covering projection $\pi\eta: G \to \Gamma\backslash {\rm SO}^+(n,1)$ over the point $\Gamma$.  
Therefore $\tilde\Gamma$ is a discrete subgroup of $G$, and so $\tilde\Gamma$ acts freely and discontinuously on $G$ by left multiplication. 

The map $\iota: {\rm SO}(n) \to \Gamma\backslash {\rm SO}^+(n,1)$, defined by $\iota(B) = \Gamma\hat B$, 
maps ${\rm SO}(n)$ homeomorphically onto the fiber of $\varepsilon: \Gamma\backslash {\rm SO}^+(n,1) \to \Gamma\backslash H^n$ 
over the point $\Gamma e_{n+1}$. 
The exact sequence
$$1 \to \Lambda\  {\buildrel i \over \longrightarrow}\  \tilde \Gamma' \to \Gamma \to 1.$$
corresponds to the exact sequence of fundamental groups
$$1 \to \pi_1({\rm SO}(n),I)\  {\buildrel \iota_\ast \over \longrightarrow}\ \pi_1(\Gamma\backslash {\rm SO}^+(n,1), \Gamma) 
\ {\buildrel \varepsilon_\ast \over \longrightarrow}\  \pi_1(\Gamma\backslash H^n,\Gamma e_{n+1}) \to 1.$$

Let $\hat\Gamma'$ be a subgroup of $\tilde\Gamma'$ of index 2. 
Then $\hat\Gamma'$ corresponds to an equivalence class of connected double covering spaces of $\Gamma\backslash {\rm SO}^+(n,1)$. 
Now $\Gamma\backslash  {\rm SO}^+(n,1)$ has a connected double covering space, corresponding to $\hat \Gamma'$, 
such that the fiber $\iota({\rm SO}(n))$ is double covered by a copy of ${\rm Spin}(n)$ 
if and only if $i^{-1}(\hat\Gamma') = \Kappa$, that is, $\Lambda\cap \hat\Gamma' = \Kappa$, by Proposition V.11.1 of \cite{M}. 
As $\Lambda \subseteq \tilde\Gamma'$, we have that $\Kappa \subseteq \hat\Gamma'$. 
The subgroups of $ \tilde\Gamma'$ of index 2 that contain $\Kappa$ correspond via $\kappa$ to the subgroups of $\tilde\Gamma$ of index 2. 

Now $\kappa$ maps $\hat\Gamma'$ onto a subgroup $\hat\Gamma$ of $\tilde\Gamma$ of index 2. 
We have that $\Lambda\cap \hat\Gamma' =\Kappa$ if and only if $\{\pm 1\}\cap \hat\Gamma = \{1\}$. 
Therefore if $\Gamma\backslash H^n$ has a spin structure, whose principal ${\rm Spin}(n)$-bundle corresponds to $\hat\Gamma'$, then $\{\pm 1\}\cap \hat\Gamma = \{1\}$. 

Conversely, suppose $\{\pm 1\} \cap \hat\Gamma = \{1\}$. 
Let $s \mapsto \hat s$ be the embedding of ${\rm Spin}(n)$ into ${\rm Spin}^+(n,1)$. 
Then ${\rm Spin}(n)$ acts on the right of $\hat\Gamma \backslash{\rm Spin}^+(n,1)$ by $(\hat \Gamma g) s = \hat \Gamma g \hat s$. 
Suppose that $(\hat \Gamma g) s = \hat \Gamma g$. 
Then $g\hat sg^{-1}$ is in $\hat\Gamma$. 
Hence 
$\eta(g\hat sg^{-1}) =  \eta(g)\eta(\hat s)\eta(g)^{-1}$
is in $\Gamma$. 
Now $\eta(g)\eta(\hat s)\eta(g)^{-1}$ fixes the point $\eta(g)e_{n+1}$ of $H^n$. 
Therefore 
$\eta(g)\eta(\hat s)\eta(g)^{-1} = I,$ 
since $\Gamma$ acts freely on $H^n$. 
Hence $\eta(\hat s) = I$, and so $\hat s = \pm 1$.  
Hence $\pm 1 = g \hat s g^{-1}$ is in $\hat\Gamma$.  
As $\{\pm 1\} \cap \hat\Gamma = \{1\}$, we must have that $\hat s = 1$, and so $s = 1$. 
Therefore ${\rm Spin}(n)$ acts freely on $\hat\Gamma \backslash{\rm Spin}^+(n,1)$. 

Define $\zeta: \hat\Gamma\backslash{\rm Spin}^+(n,1) \to \Gamma\backslash H^n$ by $\zeta(\hat\Gamma g) = \Gamma\epsilon\eta(g)$.  
If $s$ is in ${\rm Spin}(n)$, then 
$$\zeta(\hat\Gamma gs) = \Gamma(\epsilon\eta(gs)) = \Gamma \eta(g)\sigma(s)e_{n+1} = \Gamma\eta(g)e_{n+1} = \Gamma\epsilon\eta(g).$$
Therefore ${\rm Spin}(n)$ acts on each fiber of $\zeta$. 

Suppose that $\zeta(\hat\Gamma h) = \zeta(\hat\Gamma g)$. 
Then $\Gamma \eta(h) e_{n+1} = \Gamma \eta(g)e_{n+1}$. 
Hence there exists $\gamma$ in $\Gamma$ such that $\eta(g)e_{n+1} = \gamma\eta(h)e_{n+1}$. 
As $\eta(\tilde \Gamma) = \Gamma$ and $\tilde \Gamma$ is 
the disjoint union of $\hat\Gamma$ and $-\hat\Gamma$, 
we have that $\eta(\hat\Gamma) = \Gamma$. 
Hence there exists $\hat\gamma$ in $\hat\Gamma$ such that $\eta(\hat\gamma) = \gamma$. 
We have that $\eta(g)^{-1}\eta(\hat\gamma)\eta(g) e_{n+1} = e_{n+1}$, 
and so $\eta(g^{-1}\hat\gamma h)e_{n+1} = e_{n+1}$. 
Therefore $\eta(g^{-1}\hat\gamma h)$ is in $\widehat{\rm SO}(n)$. 
Hence $g^{-1}\hat\gamma h$ is in $\widehat{\rm Spin}(n)$. 
Therefore there is a $\hat s$ in $\widehat{\rm Spin}(n)$ 
such that $g^{-1}\hat\gamma h = \hat s$. 
Hence $\hat\gamma h = g\hat s$,  
and so $\hat\Gamma h = \hat \Gamma g\hat s$. 
Therefore ${\rm Spin}(n)$ acts transitively on each fiber of $\zeta$. 

Moreover $\hat\Gamma\backslash{\rm Spin}^+(n,1)$ is a a principal ${\rm Spin}(n)$-bundle with projection $\zeta$, 
since the trivialization of the principal ${\rm Spin}(n)$-bundle ${\rm Spin}^+(n,1)$ 
descends under the action of $\hat\Gamma$ to a local trivialization of $\zeta$. 

The double covering epimorphism $\eta: {\rm Spin}^+(n,1) \to {\rm SO}^+(n,1)$ 
induces a double covering projection ${\rho}: \hat\Gamma \backslash {\rm Spin}^+(n,1) \to \Gamma\backslash {\rm SO}^+(n,1)$ 
defined by ${\rm \rho}(\hat\Gamma g) = \Gamma (\eta(g))$. 
We have that 
\begin{eqnarray*}
\rho((\hat\Gamma g)s) \ \  = \ \  \rho(\hat\Gamma(g\hat s)) 
& = & \Gamma (\eta(g \hat s)) \\ 
& = & \Gamma(\eta(g)\sigma(s)) \\
& = & (\Gamma (\eta(g)))\sigma(s)\ \  = \ \ \rho(\hat\Gamma g)\sigma(s).
\end{eqnarray*}
Therefore $\rho$ is a spin structure on $\Gamma\backslash H^n$ 
such that the connected double covering space $\hat\Gamma\backslash{\rm Spin}^+(n,1)$ of $\Gamma\backslash {\rm SO}^+(n,1)$ corresponds to $\hat\Gamma'$ 
and to $\hat\Gamma$.  

Suppose $\rho': P' \to \Gamma\backslash{\rm SO}^+(n,1)$ is another spin structure on $\Gamma\backslash H^n$ 
that corresponds to $\hat\Gamma'$. 
Then $\rho$ and $\rho'$ are equivalent covering spaces of $\Gamma\backslash{\rm SO}^+(n,1)$. 
Hence there is a diffeomorphism $\xi: \hat\Gamma\backslash {\rm Spin}^+(n,1) \to P'$ such that $\rho'\xi = \rho$. 
Let $g$ be in ${\rm Spin}^+(n,1)$,  let $x = \hat\Gamma g$, and let $s$ be in ${\rm Spin}(n)$.  
Then we have that 
$$\rho'(\xi(xs)) = \rho(xs) = \rho(x)\sigma(s)$$
while on the other hand
$$\rho'(\xi(x)s) = \rho'(\xi(x))\sigma(s) = \rho(x)\sigma(s).$$
Hence $\xi(xs) = \xi(x)(\pm s)$, and so $\xi(xs)s^{-1} = \xi(x)(\pm 1)$. 
Now $s\mapsto \xi(xs)s^{-1}$ is a continuous function from ${\rm Spin}(n)$ to $\{\xi(x), \xi(x)(-1)\}$. 
As ${\rm Spin}(n)$ is connected and $1 \mapsto \xi(x)$, 
we have that $\xi(xs)s^{-1} = \xi(x)$ for all $s$ in ${\rm Spin}(n)$. 
Hence $\xi(x s) = \xi(x) s$. 
Therefore $\rho$ and $\rho'$ are equivalent spin structures on $\Gamma\backslash H^n$. 
Thus $\hat\Gamma'$ and $\hat\Gamma$ correspond to an equivalence class of spin structures on $\Gamma\backslash H^n$.

Now $\{\pm 1\} \cap \hat\Gamma = \{1\}$ if and only if $\eta$ maps $\hat \Gamma$ isomorphically onto $\Gamma$. 
Thus the set of equivalence classes of spin structures on $\Gamma\backslash H^n$ is  
in one-to-one correspondence with the set of subgroups $\hat \Gamma$ of ${\rm Spin}^+(n,1)$ 
mapped isomorphically onto $\Gamma$ by $\eta$. 
\end{proof}

\subsection{Lifting isometries} 
Let $\Gamma$ be a torsion-free discrete subgroup of ${\rm SO}^+(n,1)$, and let $f$ be an element of ${\rm SO}^+(n,1)$ 
such that $f\Gamma f^{-1} = \Gamma$.  
Then $f$ induces an orientation preserving isometry $\overline f$ of the hyperbolic space form $\Gamma\backslash H^n$    
defined by $\overline f(\Gamma x) = \Gamma fx$ by Theorem 8.1.5 of \cite{R}. 
Conversely, if $\phi$ is an orientation preserving isometry of $\Gamma\backslash H^n$, 
then there exists an element $f$ of  ${\rm SO}^+(n,1)$ such that $f\Gamma f^{-1} = \Gamma$ and $\phi = \overline f$  by Theorem 8.1.5 of \cite{R}; 
moreover, $f$ is unique up to left multiplication by an element of $\Gamma$. 

Let $\phi$ be an orientation preserving isometry of $\Gamma\backslash H^n$, and let $f$ be an element of ${\rm SO}^+(n,1)$ 
such that $f\Gamma f^{-1} = \Gamma$ and $\phi = \overline f$.  
Then $\phi$ induces a self-diffeomorphism $\phi_\star$ of  $\Gamma\backslash {\rm SO}^+(n,1)$ 
defined by $\phi_\star(\Gamma g) = \Gamma  fg$. 
Let $\eta: {\rm Spin}^+(n,1) \to {\rm SO}^+(n,1)$ be the double covering epimorphism,    
and let $\hat \Gamma$ be a subgroup of ${\rm Spin}^+(n,1)$ such that $\eta$ maps $\hat\Gamma$ 
isomorphically onto $\Gamma$. 
The isometry $\phi$ of $\Gamma\backslash H^n$ is said to {\it lift} to the spin structure $\rho: \hat\Gamma\backslash {\rm Spin}^+(n,1) \to \Gamma\backslash {\rm SO}^+(n,1)$ 
on $\Gamma \backslash H^n$ induced by $\eta$ 
if $\phi_\star$ lifts to a self-diffeomorphism $\hat \phi_\star$ of $\hat\Gamma\backslash {\rm Spin}^+(n,1)$ such that $\rho \hat \phi_\star = \phi_\star \rho$. 
If the isometry $\phi$ of $\Gamma\backslash H^n$ lifts to a spin structure on $\Gamma\backslash H^n$, 
we also say that $\phi$ {\it fixes} or {\it leaves invariant} the spin structure on $\Gamma\backslash H^n$. 

\begin{theorem}\label{thm:2.2} 
Let $\Gamma$ be a torsion-free discrete subgroup of ${\rm SO}^+(n,1)$, and let $\phi$ be an orientation preserving isometry 
of the hyperbolic space form $\Gamma\backslash H^n$. 
Let $f$ be an element of ${\rm SO}^+(n,1)$ such that $f\Gamma f^{-1} = \Gamma$ and $\phi = \overline f$.   
Let $\hat\Gamma$ be a subgroup of ${\rm Spin}^+(n,1)$ such that the double covering epimorphism $\eta: {\rm Spin}^+(n,1) \to {\rm SO}^+(n,1)$ 
maps $\hat\Gamma$ isomorphically onto $\Gamma$, and let $\hat f$ be an element of ${\rm Spin}^+(n,1)$ such that $\eta(\hat f) = f$. 
Then $\phi$ lifts to the spin structure $\rho: \hat\Gamma\backslash {\rm Spin}^+(n,1) \to \Gamma\backslash {\rm SO}^+(n,1)$ 
on $\Gamma \backslash H^n$ induced by $\eta$ if and only if $\hat f\hat \Gamma \hat f^{-1} = \hat \Gamma$. 
\end{theorem}
\begin{proof}
Suppose that $\hat f\hat \Gamma \hat f^{-1} = \hat \Gamma$.
Define $\hat \phi_\star: \hat\Gamma\backslash {\rm Spin}^+(n,1) \to \hat\Gamma\backslash {\rm Spin}^+(n,1)$ 
by $\hat \phi_\star(\hat\Gamma g) = \hat\Gamma \hat f g$. 
If $\hat\gamma$ is in $\hat\Gamma$, then $\hat\Gamma\hat f\hat \gamma g = \hat\Gamma \hat f\hat\gamma\hat f ^{-1} \hat f g = \hat \Gamma \hat f g$, 
and so $\hat \phi_\star$ is well defined. 

Let $\hat f_\ast: {\rm Spin}^+(n,1) \to {\rm Spin}^+(n,1)$ be left multiplication by $\hat f$, 
and let $\varpi: {\rm Spin}^+(n,1) \to \hat\Gamma\backslash {\rm Spin}^+(n,1)$ be the quotient map. 
Then $\hat \phi_\star \varpi = \varpi \hat f_\ast$, since
$$\hat \phi_\star \varpi(g) = \hat \phi_\star (\hat \Gamma g) = \hat\Gamma \hat f g =  \varpi(\hat f g) =  \varpi \hat f_\ast(g). $$ 
Hence $\hat \phi_\star$ is smooth, since $\varpi$ is a smooth covering projection.  
Moreover $\hat \phi_\star$ is a diffeomorphism with inverse $\hat \phi_\star^{-1}$ defined 
by $\hat \phi_\star^{-1}(\hat\Gamma g) = \hat\Gamma \hat f^{-1} g$. 

Observe that $\rho\hat\phi_\star = \phi_\star \rho$, since
$$\rho\hat\phi_\star(\hat\Gamma g) = \rho(\hat\Gamma \hat f g) = \Gamma \eta(\hat f g) = \Gamma f \eta(g) = \phi_\star (\Gamma \eta(g)) = \phi_\star(\rho(\hat\Gamma g)).$$
Therefore $\phi$ lifts to the spin structure $\rho$. 

Conversely, suppose that $\phi$ lifts to the spin structure $\rho$. 
Then $\phi_\star$ lifts to a self-diffeomorphism $\hat\phi_\star$ of $\hat\Gamma\backslash {\rm Spin}^+(n,1)$ such that $\rho \hat\phi_\star = \phi_\star \rho$. 
Multiplication by $-1$ on the right of $\hat\Gamma\backslash {\rm Spin}^+(n,1)$ is the nonidentity covering transformation of $\rho$, 
and we denote this covering transformation by left multiplication by $-1$. 
Observe that 
$$\rho \hat\phi_\star(\hat\Gamma) = \phi_\star\rho(\hat\Gamma) = \phi_\star(\Gamma) = \Gamma f.$$
Hence $\hat\phi_\star(\hat \Gamma) = \hat\Gamma(\pm \hat f) = \pm \hat\Gamma \hat f$. 
By replacing $\hat\phi_\star$ by $-\hat\phi_\star$ if necessary, we may assume that $\hat\phi_\star(\hat \Gamma) = \hat\Gamma\hat f$.
Theorem 5.1 of \cite{M} implies that 
$$(\phi_\star \rho)_\ast \pi_1(\hat\Gamma\backslash {\rm Spin}^+(n,1),\hat\Gamma) \subseteq \rho_\ast \pi_1(\hat\Gamma\backslash {\rm Spin}^+(n,1),\hat\Gamma\hat f).$$
As both $(\phi_\star)_\ast\rho_\ast \pi_1(\hat\Gamma\backslash {\rm Spin}^+(n,1),\hat\Gamma)$ and  $\rho_\ast \pi_1(\hat\Gamma\backslash {\rm Spin}^+(n,1),\hat\Gamma\hat f)$ 
are subgroups of index 2 of $\pi_1( \Gamma\backslash {\rm SO}^+(n,1),\Gamma g)$, we have that 
$$(\phi_\star)_\ast\rho_\ast \pi_1(\hat\Gamma\backslash {\rm Spin}^+(n,1),\hat\Gamma) = \rho_\ast \pi_1(\hat\Gamma\backslash {\rm Spin}^+(n,1),\hat\Gamma\hat f).$$

Let $\kappa: \widetilde{\rm Spin}\hbox{}^+(n,1) \to {\rm Spin}^+(n,1)$ be the universal covering projection considered in the proof of 
Theorem~\ref{T:spin-hyperbolic}, 
and let $\pi: {\rm SO}^+(n,1) \to \Gamma\backslash {\rm SO}^+(n,1)$ be the quotient map. 
Then $\pi\eta\kappa:  \widetilde{\rm Spin}\hbox{}^+(n,1) \to \Gamma\backslash {\rm SO}^+(n,1)$ is a universal covering projection. 
The fiber of $\pi\eta\kappa$ over the point $\Gamma$ is the subgroup $\tilde \Gamma'$ of $\widetilde{\rm Spin}\hbox{}^+(n,1)$ considered in the proof of Theorem~\ref{T:spin-hyperbolic}, 
moreover the group $\tilde \Gamma'$ acts on $\widetilde{\rm Spin}\hbox{}^+(n,1)$ by left multiplication as the group of covering transformations of $\pi\eta\kappa$. 
There is an isomorphism 
$$\chi: \tilde\Gamma' \to \pi_1( \Gamma\backslash {\rm SO}^+(n,1),\Gamma)$$
such that if $\omega: [0,1] \to \Gamma\backslash {\rm SO}^+(n,1)$ is a loop based at the point $\Gamma$, and $\tilde\omega: [0,1] \to \widetilde{\rm Spin}\hbox{}^+(n,1)$ 
is the lift of $\omega$ starting at the identity element $1$ of the group $\widetilde{\rm Spin}\hbox{}^+(n,1)$, then $\chi(\tilde\omega(1)) = [\omega]$. 

Let $\tilde f$ be an element of $\widetilde{\rm Spin}\hbox{}^+(n,1)$ such that $\kappa(\tilde f) = \hat f$. 
The fiber of $\pi\eta\kappa$ over the point $\Gamma f$ is the coset $\tilde\Gamma' \tilde f$. 
There is an isomorphism 
$$\psi: \tilde\Gamma' \to \pi_1( \Gamma\backslash {\rm SO}^+(n,1),\Gamma f)$$
such that if $\omega: [0,1] \to \Gamma\backslash {\rm SO}^+(n,1)$ is a loop based at the point $\Gamma f$, and $\tilde\omega: [0,1] \to \widetilde{\rm Spin}\hbox{}^+(n,1)$ 
is the lift of $\omega$ starting at $\tilde f$, then $\psi(\tilde\omega(1)\tilde f^{-1}) = [\omega]$. 

Let $\omega: [0,1] \to \Gamma\backslash {\rm SO}^+(n,1)$ is a loop based at the point $\Gamma$. 
Then $(\phi_\star)_\ast([\omega]) = [\phi_\star\omega]$ and $\phi_\star\omega:  [0,1] \to \Gamma\backslash {\rm SO}^+(n,1)$ is a loop based at the point $\Gamma f$. 
Let $\tilde\omega: [0,1] \to \widetilde{\rm Spin}\hbox{}^+(n,1)$ be the lift of $\omega$ starting at $1$. 
Then $\tilde f\tilde\omega :[0,1] \to \widetilde{\rm Spin}\hbox{}^+(n,1)$ is the lift of $\phi_\star\omega$ starting at $\tilde f$, which implies 
that $\psi(\tilde f \tilde\omega(1)\tilde f^{-1}) = [\phi_\star\omega]$. 

As $f\Gamma f^{-1} = \Gamma$, we have that $\tilde f \tilde\Gamma' \tilde f^{-1} = \tilde \Gamma'$. 
Let $\tilde f_\sharp$ be the automorphism of $\tilde\Gamma'$ defined by conjugating by $\tilde f$. 
Then we have that $\psi\tilde f_\sharp = (\phi_\star)_\ast \chi$. 

Let $\hat\Gamma' = \kappa^{-1}(\hat \Gamma)$. Then we have 
$$\chi(\hat\Gamma') = \rho_\ast \pi_1(\hat\Gamma\backslash {\rm Spin}^+(n,1),\hat\Gamma),$$
and 
$$\psi(\hat\Gamma') = \rho_\ast \pi_1(\hat\Gamma\backslash {\rm Spin}^+(n,1),\hat\Gamma \hat f). $$
Therefore we have that 
$$\psi\tilde f_\sharp(\hat\Gamma') = (\phi_\star)_\ast \chi(\hat\Gamma') = \psi(\hat\Gamma'),$$
and so $\tilde f_\sharp(\hat\Gamma') = \hat\Gamma'$. 
Hence $\tilde f \hat\Gamma' \tilde f^{-1} = \hat\Gamma'$. 
After applying $\kappa$, we have $\hat f\hat\Gamma \hat f^{-1} = \hat\Gamma$. 
\end{proof}


\section{Spin Groups} 

In this section, we give the formal definitions of ${\rm Spin}(n)$ and ${\rm Spin}^+(n,1)$ in terms of Clifford algebras.  
We follow the development in Chapter I of \cite{L-M}. 

\subsection{Clifford algebra} 
Let $V$ be a finite dimensional vector space over $\mathbb{K} = \realnos$ or $\complexnos$, and suppose $q$ is a nondegenerate quadratic form on $V$. 
The {\it Clifford algebra} ${\rm C\uell}(V,q)$ associated to $(V,q)$ is the associative algebra, with unit $1$, obtained from 
the free tensor algebra on $V$ by adjoining relations $v\otimes v = -q(v) 1$ for each $v$ in $V$. 
There is a natural embedding of $V$ into ${\rm C\uell}(V,q)$. 
The algebra ${\rm C\uell}(V,q)$ is generated by $V$ and the unit $1$ subject to the relations $v^2 = -q(v)1$ for each $v$ in $V$. 

The map $\alpha(v) = -v$ on $V$ extends to an algebra automorphism $\alpha$ of ${\rm C\uell}(V,q)$.  
As $\alpha^2 = id$, there is a vector space decomposition
$${\rm C\uell}(V,q) = {\rm C\uell}^0(V,q)\oplus{\rm C\uell}^1(V,q)$$
where 
$${\rm C\uell}^i(V,q) = \{x\in{\rm C\uell}(V,q): \alpha(x) = (-1)^ix\}.$$
The elements of ${\rm C\uell}^0(V,q)$ are called the {\it even} elements of ${\rm C\uell}(V,q)$. 
Note that ${\rm C\uell}^0(V,q)$ is a subalgebra of ${\rm C\uell}(V,q)$. 
The elements of ${\rm C\uell}^1(V,q)$ are called the {\it odd} elements of ${\rm C\uell}(V,q)$. 
The elements of $V$ are odd. 
The product of two elements of the same parity is even, and the product of two elements of different parities is odd.

If $v$ is in $V$ and $q(v) \neq 0$, then $v$ is invertible in ${\rm C\uell}(V,q)$ with $v^{-1} =-v/q(V)$. 
By Proposition I.2.2 of \cite{L-M}, there is a mapping ${\rm Ad}_v: V \to V$, defined by ${\rm Ad}_v(w) = vwv^{-1}$,  
and ${\rm Ad}_v = -\rho_v$ where $\rho_v$ is the reflection of $V$ in the subspace of all vectors orthogonal to $v$. 

Let ${\rm P}(V,q)$ be the of multiplicative subgroup of ${\rm C\uell}(V,q)$ generated by all $v$ in $V$ such that $q(v)\neq 0$, 
and let ${\rm O}(V,q)$ be the orthogonal group of $(V,q)$. 
Then we have a homomorphism ${\rm Ad}: {\rm P}(V,q) \to {\rm O}(V,q)$, defined by ${\rm Ad}(x)(w) = xwx^{-1}$. 

\subsection{Spin group} 
The {\it spin group} of $(V,q)$ is the subgroup of ${\rm P}(V,q)$ defined to be 
$${\rm Spin}(V,q) = \{v_1\cdots v_k : v_i \in V,\ \hbox{with}\ q(v_i) = \pm 1\ \hbox{for each}\ i,\ \hbox{and}\ k\ \hbox{even}\}.$$
All the elements of ${\rm Spin}(V,q)$ are even, and so ${\rm Spin}(V,q)$ is a subgroup of ${\rm C\uell}^0(V,q)$. 
If $\mathbb{K} = \realnos$, then the homomorphism ${\rm Ad}: {\rm P}(V,q) \to {\rm O}(V,q)$ restricts to an epimorphism 
${\rm Ad}: {\rm Spin}(V,q) \to {\rm SO}(V,q)$ 
with kernel $\{\pm 1\}$ by Theorem I.2.9 of \cite{L-M}.  

The identity map of $V$ extends to an antiautomorphism $(\  )^t: {\rm C\uell}(V,q) \to {\rm C\uell}(V,q)$ called the {\it transpose}. 
If $x, y$ are in ${\rm C\uell}(V,q)$, then $(xy)^t = y^tx^t$. 
The {\it norm} mapping $N:  {\rm C\uell}(V,q) \to {\rm C\uell}(V,q)$ is defined by $N(x) = x\alpha(x^t)$. 
Note if $v$ is in $V$, then $N(v) = q(v)$. 
Note also that $\alpha(x^t) = (\alpha(x))^t$ for all $x$ in ${\rm C\uell}(V,q)$. 
By Proposition I.2.5 of \cite{L-M}, the restriction of $N$ to the group ${\rm P}(V,q)$ gives a homomorphism $N: {\rm P}(V,q) \to {\mathbb K}^\times$. 

\begin{lemma}\label{lem:3.1} 
If ${\mathbb K} = \realnos$, then 
$${\rm Spin}(V,q) = \{x \in {\rm P}(V,q)\cap {\rm C\uell}^0(V,q) : N(x) = \pm 1\}.$$
\end{lemma}
\begin{proof}
Clearly, we have that 
$${\rm Spin}(V,q) \subseteq \{x \in {\rm P}(V,q)\cap {\rm C\uell}^0(V,q) : N(x) = \pm 1\}.$$
Suppose $x  \in {\rm P}(V,q)\cap {\rm C\uell}^0(V,q)$ and $N(x) = \pm 1$. 
Then $x = v_1\cdots v_k$ with $q(v_i)\neq 0$ for each $i$. 
As $x$ is even and each $v_i$ is odd, $k$ must be even. 
Observe that 
$$q(v_1)\cdots q(v_k) = N(v_1)\cdots N(v_k) = N(x) = \pm 1.$$
Let $\hat v_i = v_i/\sqrt{|q(v_i)|}$ for each $i$. 
Then $q(\hat v_i) = q(v_i)/|q(v_i)| = \pm 1$ for each $i$, and 
$$x = v_1\cdots v_k = \hat v_1 \cdots \hat v_k \sqrt{|q(v_1)\cdots q(v_k)|} = \hat v_1 \cdots \hat v_k.$$
Therefore $x$ is in ${\rm Spin}(V,q)$. 
\end{proof}

Now assume  $q$ is the quadratic form $x_1^2+\cdots + x_n^2$ on $V= {\mathbb K}^n$. 
If ${\mathbb K} = \realnos$, we denote ${\rm C\uell}(V,q)$ by ${\rm C\uell}(n)$ and ${\rm SO}(V,q)$ by ${\rm SO}(n)$, and ${\rm Spin}(V,q)$ by ${\rm Spin}(n)$. 
If ${\mathbb K} = \complexnos$, we denote  ${\rm C\uell}(V,q)$ by ${\complexnos\uell}(n)$. 
As discussed on p.\,20 of \cite{L-M}, both ${\rm SO}(n)$ and ${\rm Spin}(n)$ are connected Lie groups and 
the epimorphism ${\rm Ad}: {\rm Spin}(n) \to {\rm SO}(n)$ is a double covering for all $n \geq 2$.

Let $e_1, \ldots, e_n$ be the standard basis of $\realnos^n$. 
The algebra ${\rm C\uell}(n)$ is generated 
by $e_1,\ldots, e_n$ with relations $e_ie_j = -e_je_i$, for $i \neq j$ and $e_i^2 = -1$ for each $i =1,\ldots, n$. 
A basis for ${\rm C\uell}(n)$ is the set of all products $e_{i_1}\cdots e_{i_k}$ with $1 \leq i_1 < \cdots < i_k \leq n$. 
We allow the empty product which is defined to be the unit $1$. 
Therefore ${\rm C\uell}(n)$ is a $2^n$-dimensional vector space. 
A basis for ${\rm C\uell}^0(n)$ is the set of all such products of even length, 
and so ${\rm C\uell}^0(n)$ is a $2^{n-1}$-dimensional vector space. 

Now assume $q$ is the quadratic form $x_1^2 + \cdots + x_n^2 - x_{n+1}^2$ on $V= {\mathbb K}^{n+1}$. 
If ${\mathbb K} = \realnos$, we denote ${\rm C\uell}(V,q)$ by ${\rm C\uell}(n,1)$ and ${\rm SO}(V,q)$ by ${\rm SO}(n,1)$ 
and ${\rm Spin}(V,q)$ by ${\rm Spin}(n,1)$ and ${\rm P}(V,q)$ by ${\rm P}(n,1)$. 
If ${\mathbb K} = \complexnos$, we denote ${\rm C\uell}(V,q)$ by ${\complexnos\uell}(n,1)$. 
As discussed on p.\,20 of \cite{L-M}, both ${\rm SO}(n,1)$ and ${\rm Spin}(n,1)$ are Lie groups with two connected components for each $n \geq 2$.  
The connected component of ${\rm SO}(n,1)$ containing the identity is  ${\rm SO}^+(n,1)$,  
and so we denote the connected component of ${\rm Spin}(n,1)$ containing the identity by ${\rm Spin}^+(n,1)$. 
By Theorem I.2.10 of \cite{L-M}, 
${\rm Ad}: {\rm Spin}(n,1) \to {\rm SO}(n,1)$ is a double covering that restricts to a double covering ${\rm Ad}: {\rm Spin}^+(n,1) \to {\rm SO}^+(n,1)$ for all $n \geq 2$.

Let $e_1, \ldots, e_{n+1}$ be the standard basis of $\realnos^{n+1}$. 
Then $e_1, \ldots, e_{n+1}$ are Lorentz orthonormal.  
The algebra ${\rm C\uell}(n,1)$ is generated by $e_1, \ldots, e_{n+1}$ subject to the relations 
$e_ie_j = -e_je_i$, for $i \neq j$ and $e_i^2 = -1$ for each $i =1,\ldots, n$, and $e_{n+1}^2 = 1$. 
A basis for ${\rm C\uell}(n,1)$ is the set of all products $e_{i_1}\cdots e_{i_k}$ with $1 \leq i_1 < \cdots < i_k \leq n+1$. 
We allow the empty product which is defined to be the unit $1$. 
Therefore ${\rm C\uell}(n,1)$ is a $2^{n+1}$-dimensional vector space. 
A basis for ${\rm C\uell}^0(n,1)$ is the set of all such products of even length, 
and so ${\rm C\uell}^0(n,1)$ is a $2^n$-dimensional vector space. 

The algebra ${\rm C\uell}(n)$ embeds naturally into ${\rm C\uell}(n,1)$ as the subalgebra $\widehat{\rm C\uell}(n)$ generated by $e_1,\ldots, e_n$. 
We shall identity ${\rm C\uell}(n)$ with $\widehat{\rm C\uell}(n)$ via this embedding. 
Then ${\rm Spin}(n)$ is identified with the subgroup $\widehat{\rm Spin}(n)$ of ${\rm Spin}^+(n,1)$ 
that stabilizes $e_{n+1}$ under the action of ${\rm Spin}^+(n,1)$ on $\realnos^{n+1}$  
by conjugation in ${\rm C\uell}(n,1)$. 
We shall also identify ${\rm SO}(n)$ with the subgroup $\widehat{\rm SO}(n)$ of ${\rm SO}^+(n,1)$ via the embedding $B \mapsto \hat B$. 
Then ${\rm Ad}: {\rm Spin}^+(n,1) \to {\rm SO}^+(n,1)$ restricts to a double covering ${\rm Ad}: {\rm Spin}(n) \to {\rm SO}(n)$.

\begin{lemma} \label{L:spin+} 
For each integer $n \geq 2$, we have that 
$${\rm Spin}^+(n,1) = \{x \in {\rm P}(n,1)\cap {\rm C\uell}^0(n,1) : N(x) = 1\}.$$
\end{lemma}
\begin{proof}
Since $N: {\rm C\uell}(n,1) \to {\rm C\uell}(n,1)$ is a continuous function, its
restriction $N: {\rm Spin}(n,1) \to \{\pm 1\}$ is a continuous function. 
As $N(1) = 1$ and ${\rm Spin}^+(n,1)$ is connected, we have that $N({\rm Spin}^+(n,1)) = 1$. 

Observe that 
$$N(e_1e_{n+1}) = N(e_1)N(e_{n+1}) = q(e_1)q(e_{n+1}) = 1(-1) = -1.$$
Hence $e_1e_{n+1}$ is in ${\rm Spin}(n,1)$ but not in ${\rm Spin}^+(n,1)$,  
and $N$ maps the connected component of ${\rm Spin}(n,1)$ containing $e_1e_{n+1}$ to $-1$. 
Hence 
$${\rm Spin}^+(n,1) = \{x \in {\rm Spin}(n,1) : N(x) = 1\}.$$
The desired result now follows from Lemma \ref{lem:3.1}. 
\end{proof}

\begin{lemma}\label{L:spin234} 
For $n =2,3,4$, we have that 
$${\rm Spin}^+(n,1) = \{x \in {\rm C\uell}^0(n,1) : x^tx = 1\}.$$
\end{lemma}
\begin{proof}
By Proposition 16.15 of \cite{P}, 
$${\rm Spin}(n,1) = \{x \in {\rm C\uell}^0(2,1): N(x) = \pm 1\}.$$
Hence by Lemma \ref{L:spin+}, we have that
$${\rm Spin}^+(n,1) = \{x \in {\rm C\uell}^0(2,1): N(x) =  1\}.$$
Let $x$ be in $ {\rm C\uell}^0(n,1)$ such that $N(x) = 1$.  
Then $\alpha(x) = x$, and so 
$$1 = N(x) = x \alpha(x^t) = x(\alpha(x))^t = x x^t.$$
Therefore 
$${\rm Spin}^+(n,1) = \{x \in {\rm C\uell}^0(n,1): x x^t =  1\}.$$
The transpose map $(\ )^t$ restricts to an antiautomorphism of ${\rm C\uell}^0(n,1)$. 
Let $y = x^t$. 
Then
$${\rm Spin}^+(n,1) = \{y \in {\rm C\uell}^0(2,1): y^t y =  1\}.$$

\vspace{-.15in}
\end{proof}


\section{The complex spin representations}

There is a natural embedding of ${\rm C\uell}(n)$ into $\complexnos\uell(n)$ that maps the standard basis vector $e_i$ of $\realnos^n$ to 
the same vector $e_i$ of $\complexnos^n$ for each $i = 1, \ldots, n$. 
The complex algebra $\complexnos\uell(n)$ is generated by $e_1,\ldots, e_n$ subject to the same relations as in ${\rm C\uell}(n)$. 
We identify ${\rm C\uell}(n)$ with the real subalgebra of $\complexnos\uell(n)$ generated by $e_1,\ldots, e_n$.

Assume that $n$ is even and let $n = 2m$ and $k = 2^m$.  
Let $\complexnos(k)$ be the algebra of complex $k \times k$ matrices. 
There is an isomorphism $\psi: \complexnos\uell(n) \to \complexnos(k)$ of complex algebras 
by Theorem I.4.3 of \cite{L-M}. 
The {\it complex spin representation} of ${\rm Spin}(n)$ is the faithful representation 
$$\Delta_n: {\rm Spin}(n) \to \complexnos(k)$$
obtained by restricting $\psi: \complexnos\uell(n) \to \complexnos(k)$. 
The algebra $\complexnos(k)$ is central simple, and so
every automorphism of  $\complexnos(k)$ is an inner automorphism by the Skolem-Noether theorem.  
Hence $\Delta_n$ is uniquely defined up to conjugation in $\complexnos(k)$. 
Let ${\rm tr}: \complexnos(k) \to \complexnos$ be the trace map. 
Then ${\rm tr}\Delta_n: {\rm Spin}(n) \to \complexnos$ does not depend 
on the choice of the isomorphism $\psi: \complexnos\uell(n) \to \complexnos(k)$. 

Let $\omega = e_1\cdots e_n$. 
Then $e_i\omega = -\omega e_i$ for each $i=1,\ldots, n$ by Proposition I.3.3 of \cite{L-M}. 
Hence $\omega$ commutes with every element of $\complexnos\uell^0(n)$. 
Therefore $\omega$ is in the center of ${\rm Spin}(n)$. 
We have that $({\bf i}^m \omega)^2 = 1$ by Formula I.5.14 of \cite{L-M}. 
Define a matrix $C$ in $\complexnos(k)$ by $C = {\bf i}^m\Delta_n(\omega)$. 
Then we have
$$C^2 = (-1)^m\Delta_n(\omega^2) = \Delta_n((-1)^m\omega^2) = \Delta_n(1) = I.$$
Let $W^+$ and $W^-$ be the $+1$ and $-1$ eigenspaces of the matrix $C$. 
Then 
$$\complexnos^k = W^+\oplus W^-.$$
The elements of $W^+$ and $W^-$ are called the {\it positive} and {\it negative Weyl spinors}. 
We have that $\dim W^+ = \dim W^-$ by Proposition(p.\,22) of \cite{F}. 
As $\omega$ is in the center of ${\rm Spin}(n)$, the matrix $C$ commutes with 
every matrix in the image of $\Delta_n$. 
Hence every matrix in the image of $\Delta_n$ leaves both $W^+$ and $W^-$ invariant. 
Therefore we have complex representations $\Delta_n^+$ and $\Delta_n^-$ of ${\rm Spin}(n)$ into ${\rm GL}(W^+)$ and ${\rm GL}(W^-)$, respectively,  
such that 
$$\Delta_n = \Delta_n^+\oplus \Delta_n^-.$$
The representations $\Delta_n^+$ and $\Delta_n^-$ are called the {\it positive} and {\it negative complex spin representations} of ${\rm Spin}(n)$.  
The representations $\Delta_n^+$ and $\Delta_n^-$ are inequivalent irreducible  complex representations of ${\rm Spin}(n)$ by Proposition I.5.15 of \cite{L-M}.

For each positive integer $n$, there is a natural embedding of ${\rm C\uell}(n,1)$ into $\complexnos\uell(n,1)$ 
that maps the standard basis vector $e_i$ of $\realnos^{n+1}$ to 
the same vector $e_i$ of $\complexnos^{n+1}$ for each $i = 1, \ldots, n+1$. 
The complex algebra $\complexnos\uell(n,1)$ is generated by $e_1,\ldots, e_{n+1}$ subject to the same relations as in ${\rm C\uell}(n,1)$. 
We identify ${\rm C\uell}(n,1)$ with the real subalgebra of $\complexnos\uell(n,1)$ generated by $e_1,\ldots, e_{n+1}$. 

Assume that $n$ is even, and let $n = 2m$ and $k = 2^m$. Let $\omega = e_1\cdots e_n$. 
Define 
$$\rho^{\pm}: \complexnos\uell(n,1) \to \complexnos\uell(n)$$
by  $\rho^{\pm}(e_i) = e_i$ for $i=1,\ldots,n$ and $\rho^{\pm}(e_{n+1}) = \pm {\bf i}^m \omega$. 
We have that $(\pm {\bf i}^m \omega)^2 = 1$, 
and $e_i(\pm {\bf i}^m \omega) = -(\pm {\bf i}^m \omega)e_i$ for each $i=1,\ldots, n$. 
Therefore $\rho^\pm: \complexnos\uell(n,1) \to \complexnos\uell(n)$ is a homomorphism of complex algebras 
and a retraction for each choice of $\pm$. 

\begin{lemma}\label{L:rho+-} 
The algebra homomorphism $\rho^\pm$ maps $\complexnos\uell^0(n,1)$ isomorphically onto $\complexnos\uell(n)$ 
for each choice of $\pm$. 
\end{lemma}
\begin{proof}
We have that 
$$\complexnos\uell^0(n,1) = \complexnos\uell^0(n) \oplus \complexnos\uell^1(n)e_{n+1}.$$
As $\omega$ is even, we have that 
$$\rho^\pm(\complexnos\uell^1(n)e_{n+1}) = \complexnos\uell^1(n)(\pm{\bf i}^m\omega) = \complexnos\uell^1(n).$$
Hence $\rho^\pm(\complexnos\uell^0(n,1)) = \complexnos\uell(n)$.  
As $\complexnos\uell^0(n,1)$ and $\complexnos\uell(n)$ are complex vector spaces of the same dimension, 
we deduce that $\rho^\pm$ maps $\complexnos\uell^0(n,1)$ isomorphically onto $\complexnos\uell(n)$ 
for each choice of $\pm$. 
\end{proof}

Let $\rho_0^\pm: \complexnos\uell^0(n,1) \to \complexnos\uell(n)$ be the isomorphism of complex algebras, 
obtained by restricting $\rho^\pm$.  From the proof of Lemma \ref{L:rho+-}, we have that $\rho_0^- = \alpha\rho_0^+$. 
The automorphism $\alpha$ of $\complexnos\uell(n)$ is equal to conjugation by $\omega$ by Proposition I.3.3 of  \cite{L-M}. 
Therefore $\rho_0^+$ and $\rho_0^-$ differ by an inner automorphism of $\complexnos\uell(n)$. 

The {\it complex spin representation} of ${\rm Spin}^+(n,1)$ is the faithful representation 
$$\Delta_{n,1}: {\rm Spin}^+(n,1) \to  \complexnos(k)$$
obtained by restricting $\psi\rho^\pm: \complexnos\uell(n,1)\to \complexnos(k)$ for some choice of $\pm$ 
and some choice of an isomorphism $\psi: \complexnos\uell(n) \to \complexnos(k)$ 
of complex algebras. 
The representation $\Delta_{n,1}: {\rm Spin}^+(n,1) \to  \complexnos(k)$ is uniquely defined up to conjugation in $\complexnos(k)$. 
Hence the map 
$${\rm tr}\Delta_{n,1}: {\rm Spin}^+(n,1) \to \complexnos$$ 
does not depend on any of the choices made to define $\Delta_{n,1}$. 

Finally, the complex spin representation $\Delta_{n,1}: {\rm Spin}^+(n,1) \to \complexnos(k)$ restricts to 
the complex spin representation $\Delta_n: {\rm Spin}(n) \to \complexnos(k)$ 
assuming of course that we are using the same isomorphism $\psi: \complexnos\uell(n) \to \complexnos(k)$ 
to define both $\Delta_n$ and $\Delta_{n,1}$. 


\section{The complex spin representation of ${\rm Spin}^+(2,1)$}

Consider the matrices $E_1, E_2, E_3$ in $\complexnos(2)$ which are, respectively,  
$$\left(\begin{array}{cc} 0 & 1 \\ -1 & 0 \end{array}\right), \left(\begin{array}{cc} 0 & {\bf i} \\ {\bf i} & 0 \end{array}\right), 
\left(\begin{array}{cc} 1 & 0 \\ 0 & -1 \end{array}\right). $$
We have that $E^2_1 = -I, E^2_2 = -I$, and $E_1E_2 = -E_2E_1$.  
Hence the mapping $\psi:\{e_1,e_2\} \to \{E_1,E_2\}$, defined by $\psi(e_1) = E_1$ and $\psi(e_2) = E_2$, 
extends to a homomorphism $\psi: {\complexnos\uell}(2) \to \complexnos(2)$ of complex algebras. 
The matrices $I, E_1, E_2, E_1E_2$ are linearly independent in $\complexnos(2)$, 
and so $\psi: {\complexnos\uell}(2) \to \complexnos(2)$ is an isomorphism of complex algebras.

We shall work with the complex spin representation $\Delta: {\rm Spin}^+(2,1) \to \complexnos(2)$ which is the restriction 
of the homomorphism $\psi\rho^{-}: \complexnos\uell(2,1) \to \complexnos(2)$ of complex algebras. 
Note that $\psi\rho^-(e_i) = E_i$ for $i=1,2,3$. 

The conformal disk model  of the hyperbolic plane is 
$$B^2 = \{z\in \complexnos: |z| < 1\}.$$
The group of orientation preserving isometries of $B^2$ is the group ${\rm LF}(B^2)$ of linear fractional transformations of $\complexnos$ of 
the form $(az+ b)/(\overline b z +\overline a )$ 
with $a, b$ in $\complexnos$ and $|a|^2-|b|^2 = 1$. 
Define a group $G$ by
$$G = \left\{\left(\begin{array}{rr} a & b \\ \overline{b} & \overline{a} \end{array}\right): a, b \in \complexnos\ \ \hbox{and}\ \ |a|^2-|b|^2 = 1\right\}.$$
The natural map from $G$ to ${\rm LF}(B^2)$ is a double covering epimorphism. 

If $A$ is in $\complexnos(2)$, let $A^\ast$ be the conjugate transpose of $A$.  Let $J = E_3$. 

\begin{lemma}\label{L:preserveJ} 
Let $A$ be in $\complexnos(2)$ with row vectors $(a, b)$ and $(c, d)$.  Then $A^\ast J A = J$
if and only if  $|a|^2-|c|^2 = 1, |d|^2-|b|^2 = 1$, and $\overline{b}a = \overline{d}c$. 
\end{lemma}
\begin{proof} The equation $J = A^\ast J A$ is equivalent to the equation
$$ \left(\begin{array}{cc} 1 & 0 \\ 0 & -1\end{array}\right)  =
\left(\begin{array}{cc} \overline{a} & \overline{c} \\ \overline{b} & \overline{d} \end{array}\right)\left(\begin{array}{cc} a & b \\ -c & -d\end{array}\right) = 
\left(\begin{array}{cc} |a|^2-|c|^2 & \overline{a}b - \overline{c}d \\ \overline{b}a - \overline{d}c & |b|^2-|d|^2 \end{array}\right).$$

\vspace{-.2in}
\end{proof}

\begin{lemma} \label{L:SU11} 
We have that
$$G = {\rm SU}(1,1;\complexnos) = \{A \in \complexnos(2): A^\ast J A = J\ \ \hbox{and}\ \ \det A = 1\}.$$
\end{lemma}
\begin{proof}
Solving the system of equations in Lemma \ref{L:preserveJ}, together with $ad - bc = 1$, for $c$ and $d$, 
leads to the system of equations $c = \overline{b}, d = \overline{a}$, and $|a|^2 - |b|^2 = 1$. 
\end{proof}

\begin{theorem}\label{thm:5.3} 
The complex spin representation $\Delta: {\rm Spin}^+(2,1) \to \complexnos(2)$ maps ${\rm Spin}^+(2,1)$ isomorphically onto $G$. 
\end{theorem}
\begin{proof}
As $\Delta$ is faithful, it suffices to prove that $\Delta$ maps ${\rm Spin}^+(2,1)$ onto $G$. 
Let $\phi: {\rm C}\uell(2,1) \to \complexnos(2)$ be the restriction of $\psi\rho^-:\complexnos\uell(2,1) \to \complexnos(2)$. 
Then we have that 
$$\Delta({\rm Spin}^+(2,1)) = \phi({\rm Spin}^+(2,1)).$$
The real algebra $\complexnos(2)$ is an 8-dimensional real vector space. 
The eight products $E_{i_1}\cdots E_{i_k}$ with $1\leq i_1< \cdots < i_k\leq 3$, including the empty product, are linearly independent over $\realnos$. 
Hence $\phi: {\rm C\uell}(2,1) \to \complexnos(2)$ is an isomorphism of real algebras. 

By Lemma \ref{L:spin234}, we have that
$${\rm Spin}^+(2,1) = \{x \in {\rm C\uell}^0(2,1): x^t x =  1\}.$$
The map $\sigma: \complexnos(2) \to \complexnos(2)$ defined by $\sigma(A) = JA^\ast J$ 
is an antiautomorphism of the real algebra $\complexnos(2)$.  
If $A$ has row vectors $(a,b)$ and $(c,d)$, then 
$$\sigma(A) = \left(\begin{array}{rr} \overline{a} & -\overline{c} \\ -\overline{b} & \overline{d} \end{array}\right).$$
The maps $\phi(\ )^t, \sigma\phi: {\rm C\uell}(2,1) \to \complexnos(2)$ are antihomomorphisms of real algebras.
For each $i = 1,2,3$, we have that
$$\phi(e_i^t) = \phi(e_i) = E_i = \sigma(E_i) = \sigma\phi(e_i).$$
Therefore $\phi(\  )^t = \sigma\phi$. 
Hence 
$$\phi({\rm Spin}^+(2,1)) = \{A \in \phi({\rm C}\uell^0(2,1)): \sigma(A)A= I\}.$$
We have that
$$\phi({\rm C\uell}^0(2,1)) =  \left\{\left(\begin{array}{rr} a & b \\ \overline{b} & \overline{a} \end{array}\right): a, b \in \complexnos\right\}.$$
Therefore $\phi({\rm Spin}^+(2,1)) = G$ by Lemmas \ref{L:rho+-}  and \ref{L:preserveJ}. 
\end{proof}

As $\phi: {\rm C}\uell(2,1) \to \complexnos(2)$ is an isomorphism of real algebras, 
all the algebraic structure of  ${\rm C}\uell(2,1)$ is carried over to $\complexnos(2)$ by $\phi$. 
Hence the matrices $E_1,E_2, E_3$ span a 3-dimensional real vector subspace $W$ of $\complexnos(2)$
given by 
$$W = \left\{\left(\begin{array}{rr} r & z \\ -\overline{z} & -r \end{array}\right): r \in \realnos,\ z \in \complexnos\right\}.$$
Let $B$ be in $W$ with first row vector $(r,z)$. 
Define a quadratic form $f$ on $W$ by
$$f(B) = |z|^2 - r^2.$$
Let $f_2$ be the Lorentzian quadratic form on $\realnos^{2,1}$. 
Then $\phi$ restricts to an isometry $\delta:(\realnos^3, f_2) \to (W,f)$. 

Let $A$ be in $G$ and $B$ be in $W$.  Then $ABA^{-1}$ is in $W$. 
Define ${\rm Ad}_A: W \to W$ by ${\rm Ad}_A(B) = ABA^{-1}$. 
Then ${\rm Ad}_A: W \to W$ is in ${\rm O}(W,f)$. 
Define a homomorphism ${\rm Ad}: G \to {\rm O}(W, f)$ by ${\rm Ad}(A) = {\rm Ad}_A$. 

Define an isomorphism $\delta_\ast: {\rm O}(2,1) \to  {\rm O}(W,f)$ by $\delta_\ast(T) = \delta T \delta^{-1}$. 
Then $\delta_\ast({\rm SO}(2,1)) = {\rm SO}(W,f)$.
Define ${\rm SO}^+(W,f)$ to be the connected component of ${\rm SO}(W,f)$ containing the identity. 
Then $\delta_\ast({\rm SO}^+(2,1)) = {\rm SO}^+(W,f)$, since $\delta_\ast$ is a homeomorphism. 
We have that ${\rm Ad}(G) = {\rm SO}^+(W,f)$ and 
the following diagram commutes
$$\begin{array}{ccc} 
{\rm Spin}^+(2,1) & {\buildrel \Delta \over \longrightarrow} & G \\
{\rm Ad} \downarrow  &  & \phantom{{\rm Ad}}\downarrow {\rm Ad}  \\
{\rm SO}^+(2,1) & {\buildrel \delta_\ast \over \longrightarrow} & {\rm SO}^+(W,f).
\end{array}$$

Define $\eta: G \to {\rm SO}^+(2,1)$ by $\eta = \delta_\ast^{-1}{\rm Ad}$. 
Then $\eta$ is a double covering epimorphism. 
We regard ${\rm SO}^+(2,1)$ to be a matrix group. 
If $A$ is in $G$, then $\eta(A)$ is the matrix of the isometry ${\rm Ad}(A)$ of $(W,f)$ with respect to the 
basis $E_1,E_2, E_3$. 

If $z\in \complexnos$, we write $z = z_1+ z_2{\bf i}$ with $z_1, z_2\in\realnos$; in this notation we compute that
$$\eta\left(\begin{array}{rr} a & b \\ \overline{b} & \overline{a} \end{array}\right) = 
\left(\begin{array}{rrr} 1-2a_2^2+2b_1^2 & -2a_1 a_2+2 b_1 b_2 & -2 a_1 b_1 + 2 a_2 b_2 \\ 
2 a_1 a_2 + 2 b_1 b_2&1 -2a_2^2 + 2 b_2^2 & -2 a_1 b_2 - 2a_2 b_1 \\ 
-2a_1b_1-2a_2b_2 & -2a_1b_2 + 2 a_2 b_1 & 1 + 2 b_1^2 + 2 b_2^2 \end{array}\right).$$

\begin{theorem}\label{thm:5.4} 
The complex spin representation $\Delta:{\rm Spin}^+(2,1) \to \complexnos(2)$ maps ${\rm Spin}(2)$ onto the group 
$$H = \left\{\left(\begin{array}{cc} a & 0 \\ 0 & \overline{a} \end{array}\right): a \in \complexnos\ \ \hbox{and}\ \ |a| = 1\right\}.$$
\end{theorem}
\begin{proof}
The group ${{\rm Spin}}(2)$ is the stabilizer of $e_3$ under the action of ${\rm Spin}^+(2,1)$ on $\realnos^3$ by conjugation. 
Therefore $\Delta({{\rm Spin}}(2)) = \phi({{\rm Spin}}(2))$ is the stabilizer of $E_3$ under the action of $G$ on $W$ by conjugation. 

Suppose $A$ is in $G$.  As $E_3 = J$, we have 
that $AE_3A^{-1} = E_3$ if and only if $A^{-1}J A = J = A^\ast JA$, 
and so $AE_3A^{-1} = E_3$ if and only if $A^{-1} = A^\ast$. 

Now suppose that $A^{-1} = A^\ast$.  Let $A$ have row vectors $(a,b)$ and $(\overline{b},\overline{a})$. 
As $A^{-1} = JA^\ast J$, we have 
$$\left(\begin{array}{rr} \overline{a} & -b \\ -\overline{b} & a \end{array}\right) = 
\left(\begin{array}{rr} \overline{a} & b \\ \overline{b} & a \end{array}\right).$$
Hence $b = 0$. Therefore $A$ is in $H$. 

Conversely, suppose $A$ is in $H$.  
Then $A$ is in $G$ by Lemma \ref{L:preserveJ}. 
Moreover $A^{-1} = A^\ast$, and so $A^{-1}JA = J$.  Therefore $A$ stabilizes $J = E_3$. 
\end{proof}


\section{The complex spin representation of ${\rm Spin}^+(4,1)$} 

Let $\quaternions$ be the ring of quaternions. 
Every element of $\quaternions$ can be written in the form $a + b{\bf j}$ for unique $a, b$ in $\complexnos$.
There is a monomorphism $\Psi_1: \quaternions \to \complexnos(2)$ of real algebras defined by 
$$\Psi_1(a + b{\bf j}) = \left(\begin{array}{rr} a & b \\ -\overline{b} & \overline{a} \end{array}\right).$$
Let $\quaternions(2)$ be the algebra of $2\times 2$ matrices over $\quaternions$.  
There is a monomorphism $\Psi_2: \quaternions(2) \to \complexnos(4)$ of real algebras defined by 
$$\Psi_2 \left(\begin{array}{rr} a & b \\ c & d \end{array}\right) = \left(\begin{array}{rr} \Psi_1(a) & \Psi_1(b) \\ \Psi_1(c)  & \Psi_1(d) \end{array}\right).$$
If $A$ is a matrix in $\quaternions(2)$, define $\hat A = \Psi_2(A)$. 

Consider the matrices $E_1, \ldots, E_5$ in $\quaternions(2)$ which are, respectively,  
$$\left(\begin{array}{cc} 0 & 1 \\ -1 & 0 \end{array}\right), \left(\begin{array}{cc} 0 & {\bf i} \\ {\bf i} & 0 \end{array}\right), 
\left(\begin{array}{cc} 0 & {\bf j} \\ {\bf j }& 0 \end{array}\right),   \left(\begin{array}{cc} 0 & {\bf k} \\ {\bf k} & 0 \end{array}\right), 
\left(\begin{array}{cc} 1 & 0 \\ 0 & -1 \end{array}\right). $$
We have that $E^2_i = -I$ for $i =1,\ldots, 4$. 
Moreover $E_iE_j = -E_jE_i$ for each $i \neq j$. 
Hence the mapping $\psi:\{e_1,\ldots, e_4\} \to \{\hat E_1,\ldots, \hat E_4\}$ defined by $\psi(e_i) = \hat E_i$ 
extends to a homomorphism $\psi: {\complexnos\uell}(4) \to \complexnos(4)$ of complex algebras. 

The algebra $\complexnos(4)$ is a 16 dimensional complex vector space. 
The 16 products $\hat E_{i_1}\cdots \hat E_{i_k}$ with $1\leq i_1< \cdots < i_k\leq 4$, including the empty product, are linearly independent over $\complexnos$. 
Hence $\psi: {\complexnos\uell}(4) \to \complexnos(4)$ is an isomorphism of complex algebras. 

We shall work with the complex spin representation $\Delta: {\rm Spin}^+(4,1) \to \complexnos(4)$ which is the restriction 
of the homomorphism $\psi\rho^+: \complexnos\uell(4,1) \to \complexnos(4)$ of complex algebras. 
Note that $\psi\rho^+(e_i) = \hat E_i$ for $i = 1, \ldots, 5$. 

\begin{lemma} \label{L:quat} 
The complex spin representation $\Delta: {\rm Spin}^+(4,1) \to \complexnos(4)$ factors through $\quaternions(2)$, 
that is, there is a homomorphism $\overline \Delta: {\rm Spin}^+(4,1) \to \quaternions(2)$ such that $\Delta = \Psi_2 \overline\Delta$. 
\end{lemma}
\begin{proof}
The homomorphism $\rho^+: \complexnos\uell(4,1) \to \complexnos\uell(4)$ of complex algebras restricts to a homomorphism 
$\overline{\rho}^+: {\rm C}\uell(4,1) \to {\rm C}\uell(4)$ of real algebras. 
The mapping $\overline\psi: \{e_1,\ldots, e_4\} \to \{E_1,\ldots, E_4\}$, defined by $\overline\psi(e_i) = E_i$ for $i=1,\ldots, 4$,  
extends to a homomorphism $\overline\psi:{\rm C}\uell(4) \to \quaternions(2)$ of real algebras. 
Let $\overline \Delta: {\rm Spin}^+(4,1) \to \quaternions(2)$ be the restriction of $\overline\psi\overline{\rho}^+: {\rm C}\uell(4,1) \to \quaternions(2)$. 
Then $\Delta = \Psi_2 \overline\Delta$. 
\end{proof}

If $A$ is a matrix in $\quaternions(2)$, let $A^\ast$ be the conjugate transpose of $A$, 
and let $J = E_5$. 
The group G = ${\rm SU}(1,1;\quaternions)$ is defined by
$$G = \{A \in \quaternions(2) : A^\ast J A = J\}.$$
Let 
$\hat G = \Psi_2(G)$. 
Note that $\Psi_2:\quaternions(2) \to \complexnos(4)$ maps $G$ isomorphically onto $\hat G$. 

\begin{theorem} \label{T:Ghat} 
The complex spin representation $\Delta: {\rm Spin}^+(4,1) \to \complexnos(4)$ maps ${\rm Spin}^+(4,1)$ isomorphically onto $\hat G$. 
\end{theorem}
\begin{proof}
There is a homomorphism $\overline\Delta:  {\rm Spin}^+(4,1) \to \quaternions(2)$ such that $\Delta = \Psi_2\overline{\Delta}$ by Lemma \ref{L:quat}. 
Moreover $\Psi_2 : \quaternions(2) \to \complexnos(4)$ is a monomorphism of real algebras.  
Hence, it suffices to prove that $\overline\Delta$ maps ${\rm Spin}^+(4,1)$ isomorphically onto $G$. 
 
The real algebra $\quaternions(2)$ is a 16-dimensional real vector space. 
The 16 products $E_{i_1}\cdots E_{i_k}$ with $1\leq i_1< \cdots < i_k\leq 4$, including the empty product, are linearly independent over $\realnos$. 
Hence $\overline\psi: {\rm C}\uell(4) \to \quaternions(2)$ is an isomorphism of real algebras. 
The homomorphism $\overline{\rho}^+: {\rm C}\uell(4,1) \to {\rm C}\uell(4)$ of real algebras, 
restricts to an isomorphism $\overline{\rho}^+_0: {\rm C}\uell^0(4,1)\to {\rm C}\uell(4)$ of real algebras 
by the same argument as in the proof of Lemma \ref{L:rho+-}. 
Hence $\phi = \overline\psi\overline{\rho}^+$ restricts to an isomorphism $\phi_0 = \overline\psi\overline{\rho}^+_0: {\rm C}\uell^0(4,1) \to \quaternions(2)$ 
of real algebras. 
As ${\rm Spin}^+(4,1)$ is a subgroup of ${\rm C\uell}^0(4,1)$, 
the homomorphism $\phi: {\rm C}\uell(4,1) \to \quaternions(2)$ maps ${\rm Spin}^+(4,1)$ isomorphically onto a subgroup of $\quaternions(2)$.  
As $\overline\Delta$ is the restriction of $\phi$, 
It remains only to show that $\phi({\rm Spin}^+(4,1)) = G$. 

By Lemma \ref{L:spin234}, we have that
$${\rm Spin}^+(4,1) = \{x \in {\rm C\uell}^0(4,1): x^t x =  1\}.$$
The map $\sigma: \quaternions(2) \to \quaternions(2)$ defined by $\sigma(A) = JA^\ast J$ 
is an antiautomorphism of the real algebra $\quaternions(2)$.  
If $A$ has row vectors $(a,b)$ and $(c,d)$, then 
$$\sigma(A) = \left(\begin{array}{rr} \overline{a} & -\overline{c} \\ -\overline{b} & \overline{d} \end{array}\right).$$
The maps $\phi(\ )^t, \sigma\phi: {\rm C\uell}(4,1) \to \quaternions(2)$ are antihomomorphisms of real algebras.
For each $i = 1,\ldots, 5$, we have that
$$\phi(e_i^t) = \phi(e_i) = E_i = \sigma(E_i) = \sigma\phi(e_i).$$
Therefore $\phi(\  )^t = \sigma\phi$. 
As $\phi$ restricts to an isomorphism $\phi_0: {\rm C\uell}^0(4,1) \to \quaternions(2)$ of real algebras, we have that 
$$\phi({\rm Spin}^+(4,1)) = \{A \in \quaternions(2): \sigma(A)A= I\} = G.$$

\vspace{-.2in}
\end{proof}

Consider the homomorphism $\phi: {\rm C}\uell(4,1) \to \complexnos(2)$ of real algebras defined in the proof of Theorem \ref{T:Ghat}. 
Then $\phi(e_i) = E_i$ for $i = 1,\ldots, 5$. 
The matrices $E_1,\ldots, E_5$ span a 5-dimensional real vector subspace $W$ of $\quaternions(2)$ 
and so $\phi: {\rm C}\uell(4,1) \to \quaternions(2)$ restricts to a vector space isomorphism $\delta: \realnos^5 \to W$. 

\begin{lemma} \label{L:conjugate} 
If $A$ is in $G$ and $B$ is in $W$, then $ABA^{-1}$ is in $W$. 
\end{lemma}
\begin{proof} 
There is an $x$ in ${\rm Spin}^+(4,1)$ such that $\phi(x) = A$ by Theorem \ref{T:Ghat}.  Moreover there is a $v$ in $\realnos^5$ such that $\phi(v) = B$. 
Now $xvx^{-1}$ is in $\realnos^5$ by Proposition I.2.2 of \cite{L-M}.   
Hence $\delta(xvx^{-1}) = \phi(xvx^{-1}) = ABA^{-1}$ is in $W$. 
\end{proof}

We have that 
$$W = \left\{\left(\begin{array}{rr} r & q \\ -\overline{q} & -r \end{array}\right): r \in \realnos,\ q \in \quaternions\right\}.$$
Let $B$ be in $W$ with first row vector $(r,q)$. Then $B^2 = (r^2-|q|^2)I$, 
and so $B^2$ is a scalar matrix. 
Define a quadratic form $f$ on $W$ by
$$f(B) = -\textstyle{\frac{1}{2}}{\rm tr}(B^2) = |q|^2 - r^2.$$
Let $f_4$ be the Lorentzian quadratic form on $\realnos^{4,1}$. 
Observe that $\delta: \realnos^5 \to W$ satisfies $f(\delta(v)) = f_4(v)$ for all $v$ in $\realnos^5$. 
Hence $\delta$ is an isometry from $(\realnos^5, f_4)$ to $(W,f)$. 

Let $A$ be in $G$ and let $B$ be in $W$.  Then 
$$f(ABA^{-1}) =  -\textstyle{\frac{1}{2}}{\rm tr}((ABA^{-1})^2) =  -\textstyle{\frac{1}{2}}{\rm tr}(AB^2A^{-1}) =  -\textstyle{\frac{1}{2}}{\rm tr}(B^2) = f(B).$$
Hence the map ${\rm Ad}_A: W \to W$, defined by ${\rm Ad}_A(B) = ABA^{-1}$, is in ${\rm O}(W,f)$. 
Define a homomorphism ${\rm Ad}: G \to {\rm O}(W, f)$ by ${\rm Ad}(A) = {\rm Ad}_A$. 

Define an isomorphism $\delta_\ast: {\rm O}(4,1) \to  {\rm O}(W,f)$ by $\delta_\ast(T) = \delta T \delta^{-1}$. 
Then $\delta_\ast({\rm SO}(4,1)) = {\rm SO}(W,f)$. 
Define ${\rm SO}^+(W,f)$ to be the connected component of ${\rm SO}(W,f)$ containing the identity. 
Then $\delta_\ast({\rm SO}^+(4,1)) = {\rm SO}^+(W,f)$, since $\delta_\ast$ is a homeomorphism. 

\begin{theorem}\label{T:AdG} 
We have ${\rm Ad}(G) = {\rm SO}^+(W,f)$ and 
the following diagram commutes
$$\begin{array}{ccc} 
{\rm Spin}^+(4,1) & {\buildrel \overline\Delta \over \longrightarrow} & G \\
{\rm Ad} \downarrow  &  & \phantom{{\rm Ad}}\downarrow {\rm Ad}  \\
{\rm SO}^+(4,1) & {\buildrel \delta_\ast \over \longrightarrow} & {\rm SO}^+(W,f)
\end{array}$$
with horizontal maps isomorphisms. 
\end{theorem}
\begin{proof}
It suffices to prove that the following diagram commutes
$$\begin{array}{ccc} 
{\rm Spin}^+(4,1) & {\buildrel \overline\Delta \over \longrightarrow} & G \\
{\rm Ad} \downarrow  &  & \phantom{{\rm Ad}}\downarrow {\rm Ad}  \\
{\rm SO}^+(4,1) & {\buildrel \delta_\ast \over \longrightarrow} & {\rm O}(W,f).
\end{array}$$
Let $x$ be in ${\rm Spin}^+(4,1)$.  
Then we have that 
$${\rm Ad}(\overline\Delta(x)) =  {\rm Ad}(\phi(x))={\rm Ad}_{\phi(x)}.$$
Whereas 
$$\delta_\ast({\rm Ad}(x)) = \delta_\ast({\rm Ad}_x) = \delta{\rm Ad}_x \delta^{-1}.$$
Let $B$ be in $W$.  Then ${\rm Ad}_{\phi(x)}(B) = \phi(x)B \phi(x)^{-1}$. 
Whereas  
$$\delta {\rm Ad}_x \delta^{-1}(B) = \delta{\rm Ad}_x(\delta^{-1}(B)) = \delta(x\delta^{-1}(B)x^{-1}) = \delta(x) B\delta(x^{-1}) = \phi(x) B \phi(x)^{-1}.$$
Therefore ${\rm Ad}_{\phi(x)} = \delta {\rm Ad}_x \delta^{-1}$, and so ${\rm Ad}\overline\Delta = \delta_\ast{\rm Ad}$. 
\end{proof}

Define $\eta: G \to {\rm SO}^+(4,1)$ by $\eta = \delta_\ast^{-1}{\rm Ad}$. 
Then $\eta$ is a double covering epimorphism by Theorem \ref{T:AdG}. 
We regard ${\rm SO}^+(4,1)$ to be a matrix group. 
If $A$ is in $G$, then $\eta(A)$ is the matrix of the isometry ${\rm Ad}(A)$ of $(W,f)$ with respect to the 
basis $E_1,\ldots, E_5$. 

The proof of the next lemma is the same as for Lemma \ref{L:preserveJ}. 

\begin{lemma} \label{L:QpreserveJ} 
Let $A$ be in $\quaternions(2)$ with row vectors $(a, b)$ and $(c, d)$.  Then $A^\ast J A = J$ 
if and only if  $|a|^2-|c|^2 = 1, |d|^2-|b|^2 = 1$, and $\overline{b}a = \overline{d}c$. 
\end{lemma}

Define a group $H$ by 
$$H = \left\{\left(\begin{array}{cc} a & 0 \\ 0 & d\end{array}\right): a, d \in \quaternions\ \ \hbox{and}\ \ |a| = |d| = 1\right\}.$$
Let $\hat H = \Psi_2(H)$.  Note that $\Psi_2: \quaternions(2) \to \complexnos(4)$ maps $H$ isomorphically onto $\hat H$. 

\begin{theorem}\label{thm:6.6} 
The complex spin representation $\Delta:{\rm Spin}^+(4,1) \to \complexnos(4)$ maps ${\rm Spin}(4)$ isomorphically onto the group $\hat H$. 
\end{theorem}
\begin{proof}
As $\Delta = \overline\Delta\Psi_2$, it suffices to prove that $\overline\Delta({{\rm Spin}}(4)) = H$. 
The group ${{\rm Spin}}(4)$ is the stabilizer of $e_5$ under the action of ${\rm Spin}^+(4,1)$ on $\realnos^5$ by conjugation. 
Therefore $\overline\Delta({{\rm Spin}}(4)) = \phi({{\rm Spin}}(4))$ is the stabilizer of $E_5$ under the action of $G$ on $W$ by conjugation. 

Suppose $A$ is in $G$.  As $E_5 = J$, we have 
that $AE_5A^{-1} = E_5$ if and only if $A^{-1}J A = J = A^\ast JA$, 
and so $AE_5A^{-1} = E_5$ if and only if $A^{-1} = A^\ast$. 

Now suppose that $A^{-1} = A^\ast$.  Let $A$ have row vectors $(a,b)$ and $(c,d)$. 
As $A^{-1} = \sigma(A)$, we have 
$$\left(\begin{array}{rr} \overline{a} & -\overline{c} \\ -\overline{b} & \overline{d} \end{array}\right) = 
\left(\begin{array}{rr} \overline{a} & \overline{c} \\ \overline{b} & \overline{d} \end{array}\right).$$
Hence $\overline{b} = 0 = \overline{c}$, and so $b = 0 = c$. 
We have that
$$\left(\begin{array}{rr} a & 0 \\ 0 & d \end{array}\right)\left(\begin{array}{rr} \overline{a} & 0\\ 0 & \overline{d} \end{array}\right)
= \left(\begin{array}{rr} 1 & 0 \\ 0 &1 \end{array}\right).$$
Hence $|a|^2 = 1= |d|^2$, and so $|a| = 1 = |d|$. 
Therefore $A$ is in $H$. 

Conversely, suppose $A$ is in $H$.  
Then $A$ is in $G$ by Lemma \ref{L:QpreserveJ}. 
Moreover $A^{-1} = A^\ast$, and so $A^{-1}JA = J$.  Therefore $A$ stabilizes $J = E_5$. 
Thus $H$ is the stabilizer of $E_5$ under the action of $G$ on $W$ by conjugation. 
Hence $\overline\Delta({{\rm Spin}}(4)) = H$. 
\end{proof}


\section{The complex spinor bundle of a hyperbolic spin manifold}

Let $\Gamma$ be a torsion-free discrete subgroup of ${\rm SO}^+(n,1)$. 
Then $\Gamma\backslash{\rm SO}^+(n,1)$ is a principal ${\rm SO}(n)$-bundle over 
the hyperbolic space-form $\Gamma\backslash H^n$ with bundle projection 
$\varepsilon: \Gamma\backslash{\rm SO}^+(n,1) \to \Gamma\backslash H^n$ 
defined by $\varepsilon(\Gamma A) = \Gamma Ae_{n+1}$,  
and right action of ${\rm SO}(n)$ on $\Gamma\backslash{\rm SO}^+(n,1)$ defined by 
$(\Gamma A)B = \Gamma (AB)$. 

Suppose we have a spin structure on the hyperbolic space-form $\Gamma\backslash H^n$.  
Let $\hat\Gamma$ be the corresponding subgroup of ${\rm Spin}^+(n,1)$ as in Theorem~\ref{T:spin-hyperbolic}. 
Then $\hat\Gamma\backslash{\rm Spin}^+(n,1)$ is a principal ${\rm Spin}(n)$-bundle over 
the space-form $\Gamma\backslash H^n$ with bundle projection 
$\zeta: \hat\Gamma\backslash{\rm Spin}^+(n,1) \to \Gamma\backslash H^n$ 
defined by $\zeta(\hat\Gamma g) = \Gamma ({\rm Ad}(g)e_{n+1})$,  
and right action of ${\rm Spin}(n)$ on $\hat\Gamma\backslash{\rm Spin}^+(n,1)$ defined by 
$(\hat \Gamma g)s = \hat\Gamma (gs)$.

The double covering epimorphism ${\rm Ad}: {\rm Spin}^+(n,1) \to {\rm SO}^+(n,1)$ 
induces a spin structure ${\rm \rho}: \hat\Gamma \backslash {\rm Spin}^+(n,1) \to \Gamma\backslash {\rm SO}^+(n,1)$ 
on $\Gamma\backslash H^n$ defined by ${\rm \rho}(\hat\Gamma g) = \Gamma ({\rm Ad} (g))$ by Theorem~\ref{T:spin-hyperbolic}.

Assume that $n$ is even. Let $n = 2m$ and $k = 2^m$. 
Let $\Delta_n: {\rm Spin}(n) \to \complexnos(k)$ be the complex spin representation.  
We let ${\rm Spin}(n)$ act on the left of $\complexnos^k$ by $s v = \Delta_n(s) v$. 
Then ${\rm Spin}(n)$ acts freely on the right of $(\hat\Gamma\backslash {\rm Spin}^+(n,1)) \times \complexnos^k$ 
by 
$$(\hat\Gamma g, v) s =  (\hat\Gamma gs, s^{-1} v).$$
Let $S = {\rm Spin}(n)$. 
The {\it complex spinor bundle} $\SS$ of $\Gamma\backslash H^n$ with respect to the spin structure on $\Gamma\backslash H^n$ 
corresponding to the lift $\hat \Gamma$ of $\Gamma$ 
is the complex vector bundle

$$\hat\Gamma\backslash {\rm Spin}^+(n,1) \times_{\Delta_n} \complexnos^k= \big((\hat\Gamma\backslash {\rm Spin}^+(n,1))\times \complexnos^k\big)/S$$
over $\Gamma\backslash H^n$ with bundle projection $\mu$ defined by 
$$\mu((\hat\Gamma g,v)S)  = \Gamma ({\rm Ad}(g)e_{n+1}).$$
Note that the complex spinor bundle $\SS$ is the bundle associated to the principal ${\rm Spin}(n)$-bundle $\hat\Gamma\backslash {\rm Spin}^+(n,1)$ 
via the representation $\Delta_n: {\rm Spin}(n) \to \complexnos(k)$. 

Likewise the {\it positive} and {\it negative complex spinor bundles} $\SS^\pm$ of $\Gamma\backslash H^n$ with respect to the spin structure on $\Gamma\backslash H^n$ 
corresponding to $\hat \Gamma$ are the complex vector bundles 

$$\hat\Gamma\backslash {\rm Spin}^+(n,1) \times_{\Delta_n^\pm} W^\pm = \big((\hat\Gamma\backslash {\rm Spin}^+(n,1))\times W^\pm \big)/S$$
over $\Gamma\backslash H^n$ with bundle projection $\mu^\pm$ defined by 
$$\mu^\pm((\hat\Gamma g,v)S)  = \Gamma ({\rm Ad}(g)e_{n+1}).$$
The positive and negative complex spinor bundles $\SS^\pm$ are the bundles associated to the principal ${\rm Spin}(n)$-bundle $\hat\Gamma\backslash {\rm Spin}^+(n,1)$ 
via the representations 
$$\Delta_n^\pm : {\rm Spin}(n) \to {\rm GL}(W^\pm).$$
As $\complexnos^k = W^+\oplus W^-$, we have the direct sum decomposition $\SS = \SS^+\oplus\SS^-$ of complex vector bundles.

Let $\Delta_{n,1}: {\rm Spin}^+(n,1) \to \complexnos(k)$ be the complex spin representation of ${\rm Spin}^+(n,1)$ that extends $\Delta_n$. 
Define a left action of ${\rm Spin}^+(n,1)$ on $\complexnos^k$ by $g v = \Delta_{n,1}(g) v$. 
Then ${\rm Spin}^+(n,1)$ acts diagonally on the left of $H^n \times \complexnos^k$ by 
$$g(x,v) = ({\rm Ad}(g)x, gv).$$ 
The group $\hat\Gamma$ acts freely and discontinuously on $H^n \times \complexnos^k$. 
Moreover, the orbit space 
$\hat\Gamma \backslash (H^n\times \complexnos^k)$
is a complex vector bundle over $\Gamma\backslash H^n$ 
with bundle projection $\nu$ defined by $\nu(\hat\Gamma(x,v)) = \Gamma x$.  

Define a map
$$\xi: \hat\Gamma\backslash {\rm Spin}^+(n,1)\times_{\Delta_n} \complexnos^k \to \hat\Gamma \backslash (H^n\times \complexnos^k)$$
by the formula
$$\xi((\hat\Gamma g,v)S) = \hat\Gamma({\rm Ad}(g)e_{n+1}, gv).$$
\begin{theorem}\label{thm:7.1} 
For each positive even integer $n$, 
the map 
$$\xi: \hat\Gamma\backslash {\rm Spin}^+(n,1)\times_{\Delta_n} \complexnos^k \to \hat\Gamma \backslash (H^n\times \complexnos^k)$$
is a complex vector bundle equivalence from the complex spinor bundle $\SS$ of $\Gamma\backslash H^n$, 
$$\mu: \hat\Gamma\backslash {\rm Spin}^+(n,1)\times_{\Delta_n} \complexnos^k \to \Gamma\backslash H^n,$$
to $\nu: \hat\Gamma\backslash(H^n \times \complexnos^k) \to \Gamma\backslash H^n$. 
\end{theorem}
\begin{proof}
Define a map
$$\upsilon: {\rm Spin}^+(n,1)\times \complexnos^k \to H^n\times \complexnos^k$$
by $\upsilon(g,v) = ({\rm Ad}(g)e_{n+1},gv)$. 
Then $\upsilon$ is a smooth surjection. 
The compact Lie group $S = {\rm Spin}(n)$ acts freely on the right of ${\rm Spin}^+(n,1)\times \complexnos^k$ 
by $(g,v)s = (gs,s^{-1}v)$. 
The $S$-orbits are the fibers of $\upsilon$, and so $\upsilon$ induces a diffeomorphism
$$\overline\upsilon: ({\rm Spin}^+(n,1)\times \complexnos^k)/S \to H^n\times \complexnos^k$$
defined by 
$$\overline\upsilon((g,v)S) = ({\rm Ad}(g)e_{n+1},gv).$$ 

The group $\hat\Gamma$ acts on the left of $({\rm Spin}^+(n,1)\times \complexnos^k)/S$ by $\hat\gamma((g,v)S) = (\hat\gamma,v)S$.
The map $\overline\upsilon$ is $\hat\Gamma$-equivariant, and so $\hat\Gamma$ acts freely and discontinuously on $({\rm Spin}^+(n,1)\times \complexnos^k)/S$. 
The manifold $\hat\Gamma\backslash(({\rm Spin}^+(n,1)\times \complexnos^k)/S)$ is canonically diffeomorphic 
to $((\hat\Gamma\backslash {\rm Spin}^+(n,1))\times \complexnos^k)/S$, and so we have a smooth covering projection 
$$p: ({\rm Spin}^+(n,1)\times \complexnos^k)/S \to  \hat\Gamma\backslash {\rm Spin}^+(n,1)\times_{\Delta_n} \complexnos^k$$
defined by $p((g,v)S) = (\hat\Gamma g, v)S$. 

Let $q: H^n\times \complexnos^k \to \hat\Gamma\backslash (H^n \times \complexnos^k)$ be the quotient map. 
Then $q$ is a smooth covering projection. 
We have that $q \overline\upsilon = \xi p$. 
Therefore $\xi$ is a diffeomorphism.  

Observe that $\nu\xi = \mu$, since
\begin{eqnarray*}
\nu\xi((\hat\Gamma g,v)S) & = & \nu(\hat\Gamma({\rm Ad}(g)e_{n+1}, gv)) \\ 
& = & \Gamma ({\rm Ad}(g)e_{n+1}) \\ 
& = &  \zeta(\hat\Gamma g) \ \ = \ \ \mu((\hat\Gamma g,v)S)).
\end{eqnarray*}
Therefore $\xi$ is an equivalence of fiber bundles. 

Let $x$ be a point of $H^n$, and let $g$ in ${\rm Spin}^+(n,1)$ be such that ${\rm Ad}(g)e_{n+1} = x$. 
The fiber of $\mu: \hat\Gamma\backslash {\rm Spin}^+(n,1))\times_{\Delta_n} \complexnos^k \to \Gamma\backslash H^n$ 
over $\Gamma x$ is 
$$F_x(\mu) = \{(\hat\Gamma g, v)S: v\in \complexnos^k\}.$$
The mapping $i_g: \complexnos^k \to F_x(\mu)$ defined by $i_g(v) = (\hat\Gamma g,v)S$ is a linear isomorphism. 
The fiber of $\nu: \hat\Gamma\backslash(H^n \times \complexnos^k) \to \Gamma\backslash H^n$ over $\Gamma x$ is 
$$F_x(\nu) = \{\hat\Gamma(x,v): v\in \complexnos^k\}.$$
The mapping $j_x: \complexnos^k \to F_x(\nu)$ defined by $j_x(v)= \hat\Gamma(x,v)$ is a linear isomorphism. 
Let $\xi_x: F_x(\mu) \to F_x(\nu)$ be the restriction of $\xi$. 
Let $g_\ast: \complexnos^k \to \complexnos^k$ be defined by $g_\ast(v) =gv$. 
Then $g_\ast$ is a linear automorphism. 
We have that $\xi_x i_g = j_x g_\ast$, and so $\xi_x$ is a linear isomorphism. 
Hence $\xi$ is an equivalence of complex vector bundles. 
\end{proof}

Let $\phi$ be an orientation preserving isometry of $\Gamma\backslash H^n$, 
and let $f$ be an element of ${\rm SO}^+(n,1)$ such that $f\Gamma f^{-1} =\Gamma$ and $\phi = \overline f$. 
Suppose that $\phi$ lifts to the spin structure 
$\rho:  \hat\Gamma\backslash {\rm Spin}^+(n,1) \to \Gamma\backslash {\rm SO}^+(n,1)$ on $\Gamma\backslash H^n$. 
Let $\hat f$ be an element of ${\rm Spin}^+(n,1)$ such that ${\rm Ad}(\hat f) = f$. 
Then $\hat f\hat\Gamma \hat f^{-1} = \hat\Gamma$ by Theorem~\ref{thm:2.2}, and the self-diffeomorphism $\phi_\star$ of $\Gamma\backslash {\rm SO}^+(n,1)$ 
induced by $\phi$ lifts to a self-diffeomorphism $\hat\phi_\star$ of $\hat\Gamma\backslash{\rm Spin}^+(n,1)$ induced by $\hat f$. 

The diffeomorphism $\hat\phi_\star$ induces a self-diffeomorphism $\hat\phi$ of the vector bundle 
$\hat\Gamma\backslash {\rm Spin}^+(n,1)\times_{\Delta_n} \complexnos^k$ 
defined by $\hat\phi((\hat \Gamma g,v)S) = (\hat \Gamma \hat f g, v)S$. 
We have that $\mu\hat\phi = \phi \mu$, and so $\hat\phi$ maps fibers of $\mu$ to fibers of $\mu$.  
Let $x$ be a point of $H^n$, 
and let $\hat\phi_x: F_x(\mu) \to F_{fx}(\mu)$ be the restriction of $\hat\phi$. 
Then $\hat\phi_x i_g = i_{\hat f g}$, and so  $\hat\phi_x$ is a linear isomorphism. 
Hence $\hat\phi$ is a vector bundle automorphism of $\mu$. 

Likewise $\hat\phi_\star$ induces a self-diffeomorphism $\hat\phi'$ of the vector bundle $\hat\Gamma\backslash (H^n \times \complexnos^k)$ 
defined by $\hat\phi'(\hat\Gamma(x,v)) = \hat\Gamma(fx, \hat f v)$. 
We have that $\nu\hat\phi' = \phi\nu$, and so $\hat \phi'$ maps fibers of $\nu$ to fibers of $\nu$.  
Let $x$ be a point of $H^n$, and let $\hat \phi'_x: F_x(\nu) \to F_{fx}(\nu)$ be the restriction of $\hat \phi'$. 
Then $\hat \phi'_x j_x = j_{f x}\hat f_\ast$, and so  $\hat \phi'_x$ is a linear isomorphism. 
Hence $\hat \phi'$ is a vector bundle automorphism of $\nu$.

We have that $\xi\hat \phi= \hat \phi'\xi$, since
\begin{eqnarray*}
\xi\hat \phi((\hat \Gamma g,v)S) & = & \xi((\hat \Gamma \hat f g, v)S) \\ 
& = & \hat\Gamma({\rm Ad}(\hat fg)e_{n+1},\hat fgv) \\ 
& = & \hat \phi'(\hat\Gamma({\rm Ad}(g)e_{n+1}, gv)) \ \ = \ \ \hat \phi'\xi((\hat \Gamma g,v)S).
\end{eqnarray*}

\begin{theorem}\label{thm:7.2}  
Let $\Gamma$ be a torsion-free discrete subgroup of ${\rm SO}^+(n,1)$ with $n$ even, and 
let $\phi$ be an orientation preserving isometry of the hyperbolic space form $\Gamma\backslash H^n$. 
Let $f$ be an element of ${\rm SO}^+(n,1)$ such that $f\Gamma f^{-1} = \Gamma$ and $\phi = \overline f$.    
Let $\hat\Gamma$ be a subgroup of ${\rm Spin}^+(n,1)$ such that the double covering ${\rm Ad}: {\rm Spin}^+(n,1) \to {\rm SO}^+(n,1)$ 
maps $\hat\Gamma$ isomorphically onto $\Gamma$, and let $\hat f$ be an element of ${\rm Spin}^+(n,1)$ such that ${\rm Ad}(\hat f) = f$. 
Suppose that $\hat f \hat \Gamma \hat f^{-1} = \hat\Gamma$ and $\phi$ fixes the point $\Gamma x$ of $\Gamma\backslash H^n$. 
Let $\gamma$ be the element of $\Gamma$ such that $\gamma f x = x$, 
and let $\hat\gamma$ be the element of $\hat\Gamma$ such that ${\rm Ad}(\hat\gamma) = \gamma$. 
Let $\hat\phi$ be the vector bundle automorphism of the complex spinor bundle 
$\mu: \hat\Gamma\backslash {\rm Spin}^+(n,1)\times_{\Delta_n} \complexnos^k \to \Gamma\backslash H^n$ induced by $\hat f$, 
and let $\hat \phi_x$ be the restriction of $\hat\phi$ to the fiber of $\mu$ over the point $\Gamma x$. 
Then we have that
$${\rm tr}(\hat \phi_x)  = {\rm tr}(\Delta_{n,1}(\hat\gamma\hat f)).$$
\end{theorem}
\begin{proof}
We have that 
$$\hat\phi'(\hat\Gamma(x,v)) = \hat\Gamma(fx,\hat f v) = \hat\Gamma\hat\gamma(fx,\hat f v) = \hat\Gamma(\gamma fx, \hat\gamma \hat f v).$$
and so 
$$\hat\phi'_x j_x = j_{\gamma f x}(\hat\gamma \hat f)_\ast = j_x(\hat\gamma \hat f)_\ast.$$
Hence we have that 
$${\rm tr}(\hat \phi_x) = {\rm tr}(\xi_x^{-1}\hat \phi'_x\xi_x) =  
{\rm tr}(\hat \phi'_x)={\rm tr}(j_x(\hat\gamma\hat f)_\ast j_x^{-1}) = {\rm tr}((\hat\gamma\hat f)_\ast)  = {\rm tr}(\Delta_{n,1}(\hat\gamma\hat f)).$$

\vspace{-.23in}
\end{proof}

\section{The equivariant index for an isometry with only isolated fixed points}

Let $\Gamma$ be a torsion-free discrete subgroup of ${\rm SO}^+(n,1)$ with $n$ even and $n = 2m$. 
Let $\phi$ be an orientation preserving isometry of the hyperbolic space form $\Gamma\backslash H^n$. 
Suppose that $\phi$ fixes the point $P = \Gamma x$ of $\Gamma\backslash H^n$ and $P$ is an isolated fixed point. 
Then the differential $d\phi_P:{\rm T}_P(\Gamma\backslash H^n) \to {\rm T}_P(\Gamma\backslash H^n)$ is an orientation preserving isometry 
that fixes no nonzero tangent vector by the discussion on p.\,472 of \cite{A-B}. 

Let $f$ be an element of ${\rm SO}^+(n,1)$ such that $f\Gamma f^{-1} = \Gamma$ and $\phi = \overline f$.    
Let $\gamma$ be the element of $\Gamma$ such that $\gamma f x = x$. 
Then  $d\phi_P$ isometrically lifts to an orientation preserving isometry ${\rm T}_x(\gamma f): {\rm T}_x(H^n) \to {\rm T}_x(H^n)$,  
defined by ${\rm T}_x(\gamma f) y = \gamma f y$, which fixes no nonzero vector of ${\rm T}_x(H^n)$.    
Hence we may decompose ${\rm T}_x(H^n)$ into a direct sum of Lorentz orthogonal 2-planes 
$${\rm T}_x(H^n) = E_1\oplus E_2\oplus \cdots \oplus E_m$$
which are invariant under ${\rm T}_x(\gamma f)$. 
Let $\{e_k, e_k'\}$ be a Lorentz orthonormal basis of $E_k$ chosen so that the matrix $A$ in ${\rm O}^+(n,1)$, with column vectors $e_1, e_1', \ldots, e_m, e_m', x$, 
has determinant 1. 
Relative to such a basis, ${\rm T}_x(\gamma f)$ acts as a rotation by an angle $\theta_k$ in the 2-plane $E_k$ for each $k$. 
We call the resulting set of angles $\{\theta_k\}$ a {\it coherent system of angles} for $d\phi_P$. 

By Theorem 8.35 of \cite{A-B} and the discussions on p.\,20 of  \cite{A-H} and p.\,175 of \cite{S}, we have the following theorem.  

\begin{theorem}\label{T:spin number} 
Let $\Gamma$ be a torsion-free discrete subgroup of ${\rm SO}^+(n,1)$ with $n$ even and $n = 2m$. 
Let $\phi$ be an orientation preserving isometry of the hyperbolic space form $M = \Gamma\backslash H^n$ 
with only isolated fixed points $\{P\}$. 
Suppose that $M$ admits a spin structure 
and $\phi$ lifts to an automorphism $\hat \phi$, of the same order,  
of the corresponding spinor bundle $\SS$ of $M$. 
The equivariant index ${\rm Spin}(\hat \phi, M)$ is given by 
$${\rm Spin}(\hat\phi, M) = \sum \nu(P)$$
where $P$ ranges over the fixed points of $\phi$ and 
$$\nu(P) = \varepsilon(P,\hat\phi){\bf i}^m2^{-m}\prod_{k = 1}^m\csc(\theta_k/2)$$
where $\theta_1, \ldots, \theta_m$ is a coherent system of angles for $d\phi_P$ and $\varepsilon(P,\hat\phi) = \pm 1$. 

Moreover if $0 < |\theta_k| < \pi$ for each $k$, then
$${\rm tr}(\hat \phi_P) = \varepsilon(P,\hat\phi)\prod_{k = 1}^m 2\cos(\theta_k/2),$$
and so we have that 
$$\nu(P) = {\bf i}^m 2^{-m}{\rm tr}(\hat \phi_P)\prod_{k = 1}^m \csc(\theta_k).$$
\end{theorem}


\section{Some hyperbolic 2-manifolds that admit harmonic spinors} \label{S:surface} 

As explained in the introduction, our goal is to show how the equivariant index theorem can be used to show the existence of harmonic spinors. 
In this section, we use this method to prove the existence of non-zero harmonic spinors on two examples of hyperbolic surface; one is hyperelliptic and the other is not. 

\subsection{Hyperelliptic example}

Consider a regular hyperbolic decagon $P$ centered at the center $C = e_3$ of $H^2$ as in Figure \ref{F:decagon}. 
\begin{figure}[b]
{\includegraphics[width=2.4in]{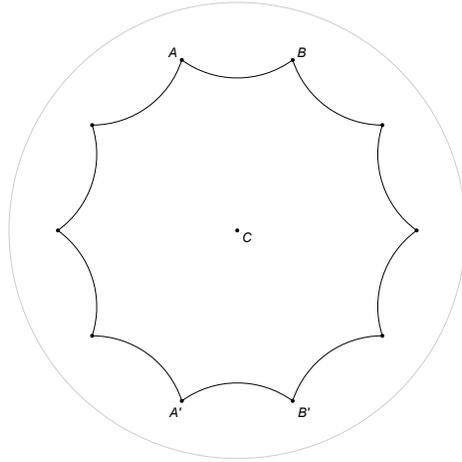}}
\caption{The regular hyperbolic decagon $P$ in $H^2$ viewed from $-e_3$.} \label{F:decagon}
\end{figure}
Let $\Gamma$ be the group generated by the 10 hyperbolic translations of $H^2$ that translate a side of $P$ to its opposite side 
and translates $P$ to an adjacent decagon. 
Then $\Gamma$ is a torsion-free discrete group of isometries of $H^2$ with fundamental polygon $P$ by Poincar\'e's theorem. 
The orbit space $M = \Gamma\backslash H^2$ is a closed orientable hyperbolic surface of genus 2. 

The side-pairing of $P$ by the generators of $\Gamma$ determines a cell decomposition of $M$ with two 0-cells, five 1-cells, and one 2-cell. 
A 0-cell corresponds to a cycle of 5 alternate vertices of $P$.  A 1-cell corresponds to a pair of opposite sides of $P$, and the 2-cell corresponds to $P$. 

Let $f$ be the rotation of $H^2$ of $2\pi/5$ about the center $C$.  
Then $f$ conjugates each generator of $\Gamma$ to another generator of $\Gamma$, and so $f\Gamma f^{-1} = \Gamma$. 
Hence $f$ induces an orientation preserving isometry $\phi$ of $M$ so that $\phi = \overline f$. 
The isometry $\phi$ has order 5 and fixes just 3 points of $M$ corresponding to $C$, 
and the two cycles of vertices of $P$ represented by the points $A$ and $B$; 
moreover, all three fixed points of $\phi$ are isolated. 

The rightmost vertex of $P$ is $(2\sqrt{2+\sqrt{5}},0,2+\sqrt{5})$, 
and the other vertices are found by rotations of multiples of $\pi/5$ about the center $C$. 
The vertices $A, B, A', B'$ are 
$$\left(\epsilon\sqrt{\textstyle{\frac{1}{2}}\big(1+\sqrt{5}\big)},\delta\sqrt{\textstyle{\frac{1}{2}}\big(15+7\sqrt{5}\big)},2+\sqrt{5}\right)$$
with $(\epsilon,\delta) = (-1,1),(1,1), (-1,-1),(1,-1)$, respectively. 
The hyperbolic translation $g_1$ that maps $A$ to $A'$ and $B$ and $B'$ is represented in ${\rm SO}^+(2,1)$ by the matrix
$$\left(\begin{array}{ccc} 1 & 0 & 0 \\ 0 & 6+3\sqrt{5} & -2\sqrt{20+9\sqrt{5}} \\ 0 & -2\sqrt{20+9\sqrt{5}} & 6 + 3\sqrt{5} \end{array}\right).$$
Let $\rho$ be the rotation of $P$ about $C$ by $\pi/5$.  
The other side-pairing maps of $P$ are then obtained by conjugating $g_1$ by multiples of $\rho$. 
These side-pairing maps generate a torsion-free discrete subgroup $\Gamma$ of ${\rm SO}^+(2,1)$ with generators $g_k = \rho^{k-1}g_1\rho^{1-k}$ for $k = 1,\ldots, 10$,  
and defining relators $g_{k+5}g_k$, for $i = 1,\ldots, 5$ and $g_7g_3g_9g_5g_1$ (corresponding to the $A$ cycle) and $g_5g_9g_3g_7g_1$ 
(corresponding to the $B$ cycle).

We lift the rotation $\rho$, 
with respect to the double covering epimorphism 
$$\eta: {\rm SU}(1,1;\complexnos) \to {\rm SO}^+(2,1)$$
defined in \S 5, to 
$$\hat\rho = \left(\begin{array}{cc} e^{{\bf i}\pi/10} & 0 \\ 0 & e^{-{\bf i}\pi/10} \end{array}\right),$$
and we lift the hyperbolic translation $g_1$ to  
$$\hat g_1 = \left(\begin{array}{cc} \textstyle{\frac{1}{2}}\big(-3-\sqrt{5}\big) & -{\bf i}\sqrt{\textstyle{\frac{1}{2}}\big(5+3\sqrt{5}\big)} \\ 
{\bf i}\sqrt{\textstyle{\frac{1}{2}}\big(5+3\sqrt{5}\big)} & \textstyle{\frac{1}{2}}\big(-3-\sqrt{5}\big) \end{array}\right).$$
We then define the lifts of the remaining side-pairing maps by conjugating $\hat g_1$ by powers of $\hat\rho$. 
These 10 elements generated a subgroup $\hat\Gamma$ of ${\rm SU}(1,1;\complexnos)$ isomorphic to $\Gamma$, since 
the defining relations of $\Gamma$ also hold in ${\rm SU}(1,1;\complexnos)$, 
and so we have a spin structure on $M$ by Theorems~\ref{T:spin-hyperbolic} and \ref{thm:5.3}. 

We have that $f = \rho^2$, and so we lift $f$ to $\hat f = -(\hat\rho)^2$. 
As $\hat f \hat \Gamma \hat f^{-1} = \hat \Gamma$, the map $\phi$ 
lifts to the spin structure defined on $M$ by Theorem \ref{thm:2.2}. 
We have that 
$$\hat f = \left(\begin{array}{cc} -e^{{\bf i}\pi/5} & 0 \\ 0 & -e^{-{\bf i}\pi/5} \end{array}\right).$$
The matrix $\hat f$ has order 5 and induces a self-diffeomorphism $\hat\phi_\star$ 
of $\hat\Gamma\backslash{\rm SU}(1,1;\complexnos)$ of order 5.  
Moreover $\hat f$ induces an automorphism $\hat\phi$ of order 5 of the corresponding spinor bundle of $M$. 
By Theorems \ref{thm:5.3} and \ref{thm:7.2}, we have that 
$${\rm tr}(\hat\phi_C) = {\rm tr}(\hat f) = -\textstyle{\frac{1}{2}}\big(1+\sqrt{5}\big).$$
Hence we have that 
$$\nu(C) = {\frac{{\bf i}}{2}}{\rm tr}(\hat\phi_C)\csc(2\pi/5) = -{\bf i}\sqrt{\textstyle{\frac{1}{10}}\big(5+\sqrt{5}\big)}.$$

Let $\gamma_A = g_7g_3$.  
Then $\gamma_A f A = A$, and $\gamma_A f$ is a rotation of $-4\pi/5$ about $A$. 
Let $ \hat\gamma_A = \hat g_7\hat g_3$. 
By Theorems  \ref{thm:5.3} and \ref{thm:7.2}, we have that
$${\rm tr}(\hat\phi_A) = {\rm tr}(\hat\gamma_A\hat f) = -\textstyle{\frac{1}{2}}\big(1-\sqrt{5}\big).$$
Hence we have that 
$$\nu(A) = {\frac{{\bf i}}{2}}{\rm tr}(\hat\phi_A)\csc(-4\pi/5) = -{\bf i}\sqrt{\textstyle{\frac{1}{10}}\big(5-\sqrt{5}\big)}.$$

Let $\gamma_B= g_6g_2$.  
Then $\gamma_B f B = B$, and $\gamma_B f$ is a rotation of $-4\pi/5$ about $B$. 
Let $ \hat\gamma_B = \hat g_6\hat g_2$. 
By Theorems  \ref{thm:5.3} and \ref{thm:7.2}, we have that
$${\rm tr}(\hat\phi_B) = {\rm tr}(\hat\gamma_B\hat f) = -\textstyle{\frac{1}{2}}\big(1-\sqrt{5}\big).$$
Hence $\nu(B) = \nu(A)$. 
By Theorem~\ref{T:spin number}, we have that 
$${\rm Spin}(\hat\phi, M)  =  \nu(A) +\nu(B) + \nu(C)  =  -{\bf i}\sqrt{\textstyle{\frac{1}{2}}\big(5+\sqrt{5}\big)}.$$
Therefore $M$ admits non-zero harmonic spinors by Proposition \ref{P:g-spin}. 
Recall that 
$${\rm Spin}(\hat\phi, M)  = {\rm tr}(\hat\phi, \mathcal{H}^+) - {\rm tr}(\hat\phi, \mathcal{H}^-).$$
Moreover  $\dim \mathcal{H}^\pm = 1$, since 
$$\dim \mathcal{H}^\pm \leq \left\lfloor (2+1)/2\right\rfloor =1.$$
Hence ${\rm Spin}(\hat\phi, M)$ is the difference of two 5th roots of unity.  
In fact, we have 
$$-{\bf i}\sqrt{\textstyle{\frac{1}{2}}\big(5+\sqrt{5}\big)} = e^{-2\pi {\bf i}/5}- e^{2\pi {\bf i}/5}.$$

The fact that $\dim \mathcal{H}^+$ is odd has a topological interpretation.  By~\cite{A}, it is equivalent to saying that $M$, together with the spin structure described above, is non-trivial in the $2$-dimensional ${\rm Spin}$ cobordism group, or in other words that its Arf invariant is non-trivial. It would be interesting to give a purely topological argument for this fact, perhaps based on the $5$-fold symmetry of the surface. 

Note that the surface $M$ has genus 2, and so $M$ is hyperelliptic. 
The theory of harmonic spinors on hyperelliptic Riemann surfaces is well understood \cite{B-S}. 
We next consider an example of a non-hyperelliptic hyperbolic surface that admits non-zero harmonic spinors 
by the combination of Proposition~\ref{P:g-spin} and Theorem~\ref{T:spin number}. 

\subsection{Non-hyperelliptic example}

Consider the Coxeter $(2,5,6)$-triangle $\Delta$ in $H^2$ shown in Figure \ref{F:nonhyperelliptic} with one vertex at the center $e_3$. 
The reflections in the sides of $\Delta$ are represented in ${\rm SO}^+(2,1)$ by the matrices
$$\left(\begin{array}{ccc} 1 & 0 & 0 \\ 0 & -1 & 0 \\ 0 & 0 & 1\end{array}\right), 
\left(\begin{array}{ccc} \frac{1}{2} & \frac{\sqrt{3}}{2} & 0 \\ \frac{\sqrt{3}}{2} & -\frac{1}{2}  & 0 \\ 0 & 0 & 1\end{array}\right),
\left(\begin{array}{ccc} -2 - \sqrt{5}& 0 & 2 \sqrt{2 + \sqrt{5}}\\  0 &  1&  0 \\ -2 \sqrt{2 + \sqrt{5}} & 0 & 2 + \sqrt{5} \end{array}\right).$$
These three reflections generate a discrete subgroup $\Gamma_0$ of ${\rm SO}^+(2,1)$. 
Consider the subgroup $\Gamma$ of $\Gamma_0$ of index 60 whose fundamental domain is the 18-sided hyperbolic polygon $P$ shown in Figure \ref{F:nonhyperelliptic}.  
The polygon $P$ is subdivided into 60 copies of $\Delta$ and is invariant under the rotation $\rho$ of $2\pi/3$ about the center $e_3$ of $P$. 
A set of generators for $\Gamma$ is a set of side-pairing transformations for $P$ where sides $A, B, C$ are paired to sides $A', B', C'$, respectively, by hyperbolic translations 
and the remaining side-pair transformations are such that the side-pairing is invariant under the rotation $\rho$. 
The orbit space $M=\Gamma\backslash H^2$ is an orientable hyperbolic surface of genus 3 by Poincar\'e's theorem. 

The rotation $\rho$ induces an orientation preserving isometry $\phi = \overline{\rho}$ of $M$ of order 3 
with exactly 5 fixed points all of which are isolated. 
The fixed points are represented by the center of $P$ and the 4 cycles of vertices of $P$. 
The quotient of  $M$ under the action of $\phi$ is an orientable hyperbolic 2-orbifold $\mathcal{O}$ of genus 0 with 5 cone points of order 3
corresponding to the fixed points of $\phi$.  
A fundamental polygon for the 2-orbifold $\mathcal{O}$ is the hyperbolic octagon $Q$ that is the third of $P$ between the two radii drawn in Figure \ref{F:nonhyperelliptic}. 
The orbifold $\mathcal{O}$ is constructed by folding $Q$ along the midline from the center and then gluing together the sides. 
Hence $M$ is trigonal, and therefore $M$ is non-hyperelliptic \cite[p.\,124]{shokurov:curves}.

The isometry $\phi$ acts as a rotation by an angle of $2\pi/3$ about the fixed point corresponding to the center of $P$, and $\phi$ acts as a rotation by an angle of $-2\pi/3$ 
about each of the fixed points corresponding to the 4 cycles of vertices of $P$. 
The isometry $\phi$ fixes a spin structure of $M$ by Proposition 5.2 of \cite{A}.  
In fact $\phi$ fixes a unique spin structure of $M$. 
We lift $\phi$ to an automorphism $\hat\phi$ of the corresponding spinor bundle of $M$ of order 3,  
and by Theorem~\ref{T:spin number}, we compute that 
$${\rm Spin}(\hat\phi, M) = -{\bf i}\sqrt{3}.$$
Therefore $M$ admits non-zero harmonic spinors by Proposition \ref{P:g-spin}.  
As described in~\cite[\S 2.2]{hitchin:spinors}, the space of harmonic spinors $\mathcal{H}^+$ on a genus $3$ surface has dimension $0$, $1$, or $2$. 
The latter case occurs only for hyperelliptic surfaces, and so ${\rm dim} \mathcal{H}^\pm = 1$ for this example. 
Therefore ${\rm Spin}(\hat\phi, M)$ is the difference of two 3rd roots of unity.  
In fact, we have
$$-{\bf i}\sqrt{3} = e^{-2\pi {\bf i}/3}- e^{2\pi {\bf i}/3}.$$
\begin{figure}[t] 
{\includegraphics[width=2.4in]{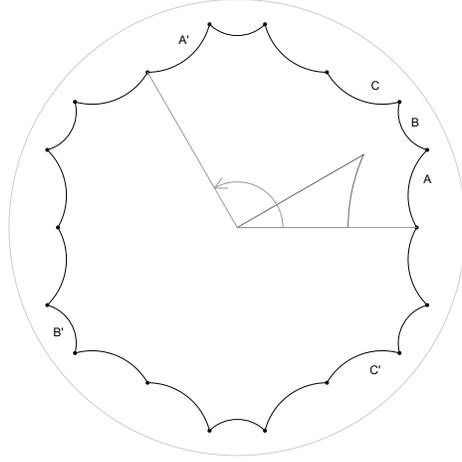}}
\caption{An 18-sided hyperbolic polygon $P$ in $H^2$ viewed from $-e_3$.}\label{F:nonhyperelliptic}
\end{figure}
\section{Harmonic spinors on the Davis hyperbolic 4-manifold} 

In this section, we combine Proposition~\ref{P:g-spin} and Theorem~\ref{T:spin number} to show that the Davis hyperbolic 4-manifold \cite{davis:hyperbolic} 
admits non-zero harmonic spinors. 
The Davis manifold $M$ is constructed geometrically by gluing together the opposite sides of a regular 120-cell $P$ in $H^4$ 
by hyperbolic translations whose axis passes through the centers of the opposite sides. 

\subsection{The 120-cell}

The 120-cell $P$ has 600 vertices, 1200 edges, 720 ridges (regular pentagons), and 120 sides (regular dodecahedra). 
The side-pairing defining the Davis manifold $M$ identifies all 600 vertices to one vertex cycle, identifies 20 edges within each cycle of edges, 
identifies 5 ridges within each cycle of ridges, and identifies 2 sides with each cycle of sides.  
Therefore, the cell structure of $P$ projects to a cell structure of $M$ consisting of one $0$-cell, 60 1-cells, 144 2-cells, 60 3-cells and one 4-cell. 

Define
$$\tau = (1+\sqrt{5})/2 \quad \hbox{and} \quad \kappa = \sqrt{1+3\tau}.$$
We will work with the regular 120-cell $P$ centered at the center $e_5$ of $H^4$ which is barycentrically subdivided by a $(5,3,3,5)$ Coxeter 4-simplex $\Delta$ in $H^4$ 
whose vertices are given by the equations
\begin{eqnarray*}
v_1 & = & \big((2+3\tau)\kappa, (1+\tau)\kappa,0,\kappa, 5+8\tau\big), \\
v_2 & = & \big((1+\tau)\kappa, 0,0,2+3\tau\big), \\
v_3 & = & \big(\tau\kappa,(-1+2\tau)\kappa/5,(3-\tau)\kappa/5,0,1+2\tau\big), \\
v_4 & = & (\kappa, 0, 0, 0, 1+\tau), \\
v_5 & = & (0, 0, 0, 0, 1).
\end{eqnarray*}
The vertices $v_1,\ldots, v_5$ are, respectively, a vertex, center of an edge, center of a ridge, center of a side, and the center of the 120-cell $P$. 
The Lorentz normal vector $s_i$ of the side of $\Delta$ opposite the vertex $v_i$ is given by the equations
\begin{eqnarray*}
s_1 & = & (0,0,0,-1,0), \\
s_2 & = & \big(0,(1-\tau)/2,1/2,\tau/2,0), \\
s_3 & = & (0,0,-1,0,0), \\
s_4 & = & \big((1-\tau)/2,\tau/2,1/2,0,0\big), \\
s_5 & = & (1+\tau,0 ,0 ,0, \kappa).
\end{eqnarray*}

\subsection{Davis manifold construction}
The 120 side-pairing maps for the 120-cell $P$ defining the Davis manifold $M$ are represented by symmetric Lorentzian $5\times 5$ matrices of the form 

$$\left(\begin{array}{ccccc} 
1+\frac{a_1^2}{1+a_5} & \frac{a_1a_2}{1+a_5} & \frac{a_1a_3}{1+a_5} & \frac{a_1a_4}{1+a_5} & a_1 \\
 \frac{a_1a_2}{1+a_5} &1+\frac{a_2^2}{1+a_5} & \frac{a_2a_3}{1+a_5} & \frac{a_2a_4}{1+a_5} & a_2 \\
 \frac{a_1a_3}{1+a_5} & \frac{a_2a_3}{1+a_5} &1+\frac{a_3^2}{1+a_5} & \frac{a_3a_4}{1+a_5} & a_3 \\
 \frac{a_1a_4}{1+a_5} & \frac{a_2a_4}{1+a_5} & \frac{a_3a_4}{1+a_5} &1+\frac{a_4^2}{1+a_5} & a_4 \\
 a_1 & a_2 & a_3 & a_4 & a_5
\end{array}\right).$$
The vector $(a_1, a_2, a_3, a_4, a_5)$ is the center of the 120-cell adjacent to $P$ that is the translated image of $P$ by 
the corresponding side-pairing map of $P$. 
All the 120 matrices representing the side-pairing maps have $a_5 = 3+6\tau$. 

We call $(a_1,a_2,a_3,a_4)$ the {\it direction vector} for the corresponding side-pairing map of $P$. 
The direction vectors for the 120 side-pairing maps are the 60 vectors listed in Table 1 of \cite {ratcliffe-tschantz:davis} together with their negatives. 
The 120 side-pairing maps $g_1, \ldots, g_{120}$ of $P$ are ordered so that $g_1,\ldots, g_{60}$ have the same order as their direction vectors in Table 1 
and $g_{61}, \ldots, g_{120}$ are ordered so that $g_{121-i}$ has direction vector equal to the negative of the $i$th vector in Table 1.

By Poincar\'e's fundamental polyhedron theorem, the 120 matrices representing the side-pairing maps $g_1, \ldots, g_{120}$ generate 
a discrete subgroup $\Gamma$ of ${\rm SO}^+(4,1)$ such that $M = \Gamma\backslash H^4$. 
Moreover $\Gamma$ has a presentation with 120 generators $x_1, \ldots, x_{120}$, 
corresponding to the 120 side-pairing maps of $P$ in the same order above, 60 side-pairing relations $x_ix_{121-i} = 1$, with $i = 1,\ldots, 60$, 
and 144 ridge cycle relations $x_ix_jx_kx_\ell x_m = 1$ where $(i,j,k,\ell, m)$ is one of the 5-tuples listed in Table 2 of \cite {ratcliffe-tschantz:davis}.

\subsection{Symmetric spin structure}

We next describe a symmetric spin structure on the Davis manifold $M$ with which we will work. 
The direction vector of the first side-pairing map $g_1$ is 
$$((2+2\tau)\kappa, 0,0,0).$$
The matrix in ${\rm SO}^+(4,1)$ representing $g_1$ lifts, with respect to the double covering 
$$\eta: {\rm SU}(1,1; \mathbb{H}) \to {\rm SO}^+(4,1),$$ 
to the real $2\times 2$ matrix 
$$\hat g_1 = \left(\begin{array}{cc}
-1-\tau & \kappa \\
\kappa & -1-\tau
\end{array}\right).$$

The four reflections $\rho_1, \ldots, \rho_4$ in the sides of $\Delta$ that contain the vertex $v_5 = e_5$ 
generated the group ${\rm Sym}(P)$ of symmetries of $P$, which is a group of order 14,400. 
The group ${\rm SU}(1,1; \mathbb{H})$ is a subgroup of index 2 in the group 
$${\rm U}(1,1; \mathbb{H}) = \{A \in \mathbb{H}(2): A^\ast J A = \pm J\}$$
and $\eta: {\rm SU}(1,1; \mathbb{H}) \to {\rm SO}^+(4,1)$ extends to a double covering epimorphism
$$\tilde\eta :{\rm U}(1,1; \mathbb{H}) \to {\rm O}^+(4,1)$$
such that $\rho_1, \ldots, \rho_4$ lift, respectively, to $2\times 2$ matrices $R_1, \ldots, R_4$ of the form 
$$\left(\begin{array}{rr}
r & q \\
-\overline{q} & -r
\end{array}\right)$$
with $r \in \realnos, q \in \mathbb{H}$, and $|q|^2-r^2 = 1$, and if $q = a+ b{\bf i} + c{\bf j} + d{\bf k}$, 
then $(a,b,c,d,r)$ is the corresponding normal vector $s_i$ given above. 

We will not describe the extension $\tilde\eta$, since it involves ${\rm Pin}$ groups which we have avoided discussing. 
It suffices to say that the matrices $R_1, \ldots, R_4$ generate a group $\widetilde{\rm Sym}(P)$ of order 28,800 
and the subgroup $\widetilde{\rm Sym}(P)_0$ consisting of even products of $R_1, \ldots, R_4$ has order 14,400,  
and $\eta: {\rm SU}(1,1; \mathbb{H}) \to {\rm SO}^+(4,1)$ maps  $\widetilde{\rm Sym}(P)_0$ onto 
the subgroup ${\rm Sym}(P)_0$ of ${\rm Sym}(P)$ of index 2 consisting of the orientation preserving symmetries of $P$. 

The orbit of $\hat g_1$ under the action of $\widetilde{\rm Sym}(P)_0$ by conjugation consists of 120 matrices $\hat g_1, \ldots, \hat g_{120}$
such that $\eta(\hat g_i) = g_i$ for each $i = 1,\ldots, 120$. 
The matrices $\hat g_1, \ldots, \hat g_{120}$ satisfy the same relations as $g_1,\ldots, g_{120}$. 
Hence $\eta: {\rm SU}(1,1; \mathbb{H}) \to {\rm SO}^+(4,1)$ maps the group $\hat\Gamma$ generated by $\hat g_1, \ldots, \hat g_{120}$ 
isomorphically onto $\Gamma$. 
Therefore $M$ has a spin structure that is invariant under ${\rm Sym}(P)_0$ by Theorems \ref{T:spin-hyperbolic}, \ref{thm:2.2}, and \ref{T:AdG}. 
This spin structure is unique, since $-\hat g_1, \ldots, -\hat g_{120}$ do not satisfy the same relations as $g_1,\ldots, g_{120}$ 
since the ridge cycle relations are of odd length. 

\subsection{Analysis of isometries of $M$ of odd order}
Let ${\rm Isom}_0(M)$ be the group of orientation preserving isometries of the Davis manifold $M$ (see \cite{ratcliffe-tschantz:davis}), 
and let $\phi$ be an element of ${\rm Isom}_0(M)$ of odd order with only isolated fixed points. 
The possible orders of $\phi$ are 3, 5, and 15. 
If $\phi$ has order 3 or 15, then $\phi$ has only 2 fixed points, namely, 
the points represented by the center of $P$ and the cycle of vertices of $P$. 

The isometry $\phi$ lifts to an automorphism $\hat\phi$, of the same order, of the spinor bundle of $M$,  
with respect to the symmetric spin structure on $M$ described above.  
If the order of $\phi$ is 3, then $\phi$ is unique up to conjugation in ${\rm Isom}_0(M)$, 
and we compute that ${\rm Spin}(\hat\phi,M) = 0$. 
If the order of $\phi$ is 5, then $\phi$ lies in one of 3 possible conjugacy classes of ${\rm Isom}_0(M)$. 
The constant value of ${\rm Spin}(\hat\phi,M)$ on these conjugacy classes is $0, -5\sqrt{5}$, and $5\sqrt{5}$. 
The latter two conjugacy classes determine a single conjugacy class of subgroups of ${\rm Isom}_0(M)$ of order 5. 
If the order of $\phi$ is 15, then $\phi$ lies in one of 2 possible conjugacy classes of ${\rm Isom}_0(M)$. 
The constant value of ${\rm Spin}(\hat\phi,M)$ on these conjugacy classes is $-\sqrt{5}$ and $\sqrt{5}$. 
These two conjugacy classes determine a single conjugacy class of subgroups of ${\rm Isom}_0(M)$ of order 15. 
If the order of $\phi$ is 15, then each power $\phi^k$ for $k =1,\ldots, 14$ has isolated fixed points, 
and ${\rm Spin}(\hat\phi^k, M) = \pm 5 \sqrt{5}$ for $k =3,6, 9, 12$.  


\subsection{Order 15 example}\label{S:10.5}
Let $f$ be the symmetry of $P$ of order 15 represented in ${\rm SO}^+(4,1)$ by the matrix
$$\left(\begin{array}{ccccc}
0 & -1 & 0 & 0 & 0 \\
\frac{\tau}{2} & 0 & -\frac{1}{2} & -\frac{1}{2}+\frac{\tau}{2} & 0 \\
-\frac{1}{2}+\frac{\tau}{2} & 0 & \frac{\tau}{2} & \frac{1}{2} & 0 \\
-\frac{1}{2} & 0 & \frac{1}{2}-\frac{\tau}{2} & \frac{\tau}{2} & 0 \\
0 & 0 & 0 & 0 & 1
\end{array}\right).$$
Then $f$ conjugates each generator of $\Gamma$ to another generator of $\Gamma$, and so $f\Gamma f^{-1} =\Gamma$. 
Hence $f$ induces an orientation preserving isometry $\phi$ of $M$ so that $\phi = \overline f$. 
The isometry $\phi$ has order 15 and fixes just 2 points $C$ and $A$ of $M$, with $C$ represented by the center $e_5$ of $P$, and 
$A$ represented by the cycle of vertices of $P$; 
moreover the two fixed points of $\phi$ are isolated. 
The angles of rotation of $\phi$ are $-8\pi/15$ and $2\pi/15$ about $C$, and $-14\pi/15$ and $4\pi/15$ about $A$. 

The symmetry $f$ lifts, with respect to $\eta: {\rm SU}(1,1; \mathbb{H}) \to {\rm SO}^+(4,1)$, to the matrix
$$\hat f = \left(\begin{array}{cc}
-\frac{\tau}{2}-\frac{1}{2}{\bf i}+\big(\frac{1}{2}-\frac{\tau}{2}\big){\bf j} & 0 \\
0 & -\frac{1}{2}+\frac{\tau}{2}{\bf i} +\big(\frac{1}{2}-\frac{\tau}{2}\big){\bf k}
\end{array}\right).$$
The matrix $\hat f$ has order 15 and induces a self-diffeomorphism $\hat \phi_\star$ of $\hat\Gamma\backslash {\rm SU}(1,1;\mathbb{H})$. 
Moreover $\hat f$ induces an automorphism $\hat\phi$ of order 15 of the corresponding spinor bundle of $M$. 
By Theorems \ref{T:Ghat} and \ref{thm:7.2}, we have that 
$${\rm tr}(\hat\phi_C) = {\rm tr}(\Psi_2(\hat f)) = -1 -\tau.$$
Hence we have that 
$$\nu(C) = {\bf i}^22^{-2}{\rm tr}(\hat\phi_C)\csc(-8\pi/15)\csc(2\pi/15) = -\tau.$$
Let $\gamma_A = g_2 g_{87} g_{46} g_{71} g_{79} g_{10} g_{115} g_{16} g_{107} g_{15} g_{73}$. 
Then $\gamma_Af v_1 = v_1$.  Let $\hat\gamma_A$ be the product of the corresponding lifts $\hat g_i$. 
By Theorems \ref{T:Ghat} and \ref{thm:7.2}, we have that 
$${\rm tr}(\hat\phi_A) = {\rm tr}(\Psi_2(\hat\gamma_A\hat f)) = -2 +\tau.$$
Hence we have that 
$$\nu(A) = {\bf i}^22^{-2}{\rm tr}(\hat\phi_A)\csc(-14\pi/15)\csc(4\pi/15) = 1-\tau.$$
By Theorem \ref{T:spin number}, we have that
$${\rm Spin}(\hat\phi, M) = \nu(A) + \nu(C) = 1-2\tau = -\sqrt{5}.$$
Thus the Davis manifold $M$ admits non-zero harmonic spinors by Proposition \ref{P:g-spin}, 
and so $\dim(\mathcal{H}^+) \geq 1$.  
We will obtain a better lower for the dimension of the space $\mathcal{H}^+$ 
of positive harmonic spinors on $M$ from our next example. 


\subsection{Order 5 example}
We maintain the notation from the previous example. 
Then $f^3$ has order $5$ and is represented in ${\rm SO}^+(4,1)$ by the matrix
$$\left(\begin{array}{ccccc}
-\frac{1}{2}+\frac{\tau}{2} & \frac{\tau}{2} & \frac{1}{2} & 0 & 0 \\
-\frac{\tau}{2} & -\frac{1}{2}+\frac{\tau}{2}  & 0 & -\frac{1}{2}& 0 \\
-\frac{1}{2} & 0 & -\frac{1}{2}+\frac{\tau}{2} & \frac{\tau}{2} & 0 \\
0 & \frac{1}{2}  & -\frac{\tau}{2} & -\frac{1}{2}+\frac{\tau}{2} & 0 \\
0 & 0 & 0 & 0 & 1
\end{array}\right).$$
The corresponding isometry $\phi^3$ of $M$ has order 5 
and fixes exactly 26 points $C, A$ and $B_1,\ldots, B_{24}$ of $M$, with $C$ represented by the center $e_5$ of $P$, and 
$A$ represented by the cycle of vertices of $P$, and $B_1, \ldots, B_{24}$ represented by 24 ridge center cycles 
one of which is the cycle of $v_3$;   
moreover all the fixed points of $\phi^3$ are isolated. 
The angles of rotation of $\phi^3$ are $2\pi/5$ and $2\pi/5$ about $C$, and $-4\pi/5$ and $4\pi/5$ about $A$ 
and $-2\pi/5$ and $4\pi/5$ about $B_i$ for each $i$. 

For each side $S$ of $P$ there are two opposite sides of $S$ whose centers $\alpha$ and $\beta$ 
represent two of the fixed points of $\phi^3$. 
In fact, the 2-dimensional cross-section of $P$ passing through $\alpha$ and $\beta$ and the center $C$ of $P$ 
is the decagon in Figure \ref{F:decagon} with the same identification pattern and the same order 5 symmetry and 
whose cycles of vertices are the ridge center cycles of $\alpha$ and $\beta$ (labeled $A$ and $B$ in Figure \ref{F:decagon}). 

The symmetry $f^3$ lifts, with respect to $\eta: {\rm SU}(1,1; \mathbb{H}) \to {\rm SO}^+(4,1)$, to the matrix
$$\hat f^3 = \left(\begin{array}{cc}
-\frac{1}{2}+\frac{\tau}{2}-\frac{\tau}{2}{\bf i} -\frac{1}{2}{\bf j} & 0 \\
0 & 1\end{array}\right).$$
The matrix $\hat f^3$ has order 5 and induces the self-diffeomorphism $\hat \phi_\star^3$ of $\hat\Gamma\backslash {\rm SU}(1,1;\mathbb{H})$. 
Moreover $\hat f^3$ induces the automorphism $\hat\phi^3$ of order 5 of the corresponding spinor bundle of $M$. 
By Theorems \ref{T:Ghat} and \ref{thm:7.2}, we have that 
$${\rm tr}((\hat\phi^3)_C) = {\rm tr}(\Psi_2(\hat f^3)) = 1 +\tau.$$
Hence we have that 
$$\nu(C) = {\bf i}^22^{-2}{\rm tr}((\hat\phi^3)_C)\csc(2\pi/5)\csc(2\pi/5) = -\frac{2}{5}-\frac{\tau}{5}.$$
Let $\gamma_A = g_3 g_{86} g_{46} g_{73} g_{79} g_{10} g_{113} g_{18} g_{107} g_{15} g_{71}g_{89}$. 
Then $\gamma_Af ^3v_1 = v_1$.  Let $\hat\gamma_A$ be the product of the corresponding lifts $\hat g_i$. 
By Theorems \ref{T:Ghat} and \ref{thm:7.2}, we have that 
$${\rm tr}((\hat\phi^3)_A) = {\rm tr}(\Psi_2(\hat\gamma_A\hat f^3)) = 2 -\tau.$$
Hence we have that 
$$\nu(A) = {\bf i}^22^{-2}{\rm tr}((\hat\phi^3)_A)\csc(-4\pi/5)\csc(4\pi/5) = \frac{3}{5}-\frac{\tau}{5}.$$
Let $B$ be one of the points $B_1, \ldots, B_{24}$.  
We found that ${\rm tr}((\hat\phi^3)_B) = -1$ and 
$$\nu(B) = {\bf i}^22^{-2}{\rm tr}((\hat\phi^3)_B)\csc(-2\pi/5)\csc(4\pi/5) = \frac{1}{5}-\frac{2\tau}{5}.$$
By Theorem \ref{T:spin number}, we have that
$${\rm Spin}(\hat\phi^3, M) = \nu(A) + 24\hspace{.01in}\nu(B)+ \nu(C) =  5 - 10\tau = -5\sqrt{5}.$$
Recall that 
$${\rm Spin}(\hat\phi^3, M) = {\tr}(\hat\phi^3,\mathcal{H}^+) -  {\tr}(\hat\phi^3,\mathcal{H}^-).$$
Let $d = \dim\mathcal{H}^\pm$.  Then ${\rm tr}(\hat\phi^3,\mathcal{H}^\pm)$ is a sum of $d$ 5th roots of unity. 
The number $-5\sqrt{5}$ cannot be written as the sum of nine 5th roots of unity minus the sum 
of nine 5th roots of unity, and so $d \geq 10$.  We will obtain the lower bound $d \geq 10$ in a more elegant manner 
in the next section (see Corollary \ref{C:10.2}). 

\subsection{Spinor-index}
Let $G$ be the cyclic group generated by the order 15 automorphism $\hat\phi$ in \S \ref{S:10.5}. 
Then $G$ acts on ${\mathcal H}^+$ and ${\mathcal H}^-$. 
We get two characters of $G$ whose difference in the representation ring $R(G)$ is the spinor-index ${\rm Spin}(G,M)$ of the action. 
The value of ${\rm Spin}(G,M)$ at an element $g$ of $G$ is 
$${\rm Spin}(g,M)= {\rm tr}(g,\mathcal{H}^+) - {\rm tr}(g,\mathcal{H}^-).$$
We have that $R(G) \cong \mathbb{Z}[x]/(x^{15}-1)$. 
The next theorem neatly summarizes our computations concerning the action of $G$ 
on $\mathcal{H} = {\mathcal H}^+\oplus {\mathcal H}^-$.

\begin{theorem}\label{T:10.1} 
The spinor-index ${\rm Spin}(G,M)$ in $R(G)$ corresponds to the coset $[p(x)]$ in $\mathbb{Z}[x]/(x^{15}-1)$ where
$$p(x) = 2x^2+x^3+2x^7+2x^8+x^{12}+2x^{13}-2x-2x^4-x^6-x^9-2x^{11}-2x^{14}$$
and 
$${\rm Spin}(\hat\phi^k,M) = p(e^{2k\pi{\bf i}/15})\ \ \hbox{for}\ \ k = 1, 2, \ldots, 15.$$
\end{theorem} 
\begin{proof}
Our computations show that ${\rm Spin}(\hat\phi^k,M) = p(e^{2k\pi{\bf i}/15})$ for $k = 1, 2, \ldots, 15$.
Every coset in $\mathbb{Z}[x]/(x^{15}-1)$ is represented by a unique polynomial in $\mathbb{Z}[x]$ of degree at most 14. 
Suppose $q(x)$ is another polynomial in $\mathbb{Z}[x]$ of degree at most 14 such that 
$${\rm Spin}(\hat\phi^k,M) = q(e^{2k\pi{\bf i}/15})\ \ \hbox{for}\ \ k = 1, 2, \ldots, 15.$$
Then $p(x)-q(x)$ has $e^{2k\pi{\bf i}/15}$ for $k = 1, 2, \ldots, 15$ as roots, and so $p(x) - q(x) = 0$.  
Thus $p(x) = q(x)$, and therefore  ${\rm Spin}(G,M)$ corresponds to the coset $[p(x)]$ in $\mathbb{Z}[x]/(x^{15}-1)$. 
\end{proof}

The next corollary completes the proof of Theorem \ref{T:davis}.
\begin{corollary} \label{C:10.2}%
The (complex) dimension of $\mathcal{H}$ is at least $20$. 
\end{corollary}
\begin{proof} If $\dim {\mathcal H}^+$ were less than $10$, 
then ${\rm Spin}(G, M)$ would be represented by a polynomial in $\mathbb{Z}[x]$ 
of degree at most 14 whose positive coefficients sum to less than 10, which is not the case by Theorem \ref{T:10.1}. 
\end{proof}

Our computations are consistent with $\dim {\mathcal H} = 20$. 
We end the paper with the following intriguing question. 

\subsection*{Question}  What is the dimension of the space ${\mathcal H}$ of harmonic spinors on the Davis hyperbolic 
4-manifold $M$?
\vspace{.15in}


\end{document}